\documentclass{article}
\usepackage[utf8]{inputenc}

\title{Excision of Skein Categories and Factorisation Homology}
\author{Juliet Cooke}
\date{October 2019}

\usepackage{lmodern}
\usepackage{amsmath, amsfonts, amssymb, amsthm, calligra, mathrsfs, mathtools, tikz-cd, eqparbox, faktor, stmaryrd}
\usepackage{xstring}
\newcommand{\myincludesvg}[3]{\includegraphics[#1]{svg-extract/#3_svg-raw-extract.pdf}}
\newcommand{\myincludesvggroup}[4]{\includegraphics[#1]{svg-extract/#3-#4_svg-raw-extract.pdf}}
\allowdisplaybreaks
\usepackage[bookmarks]{hyperref}
\usepackage[noabbrev, capitalise]{cleveref}
\usepackage{subfig}

\usepackage[style=alphabetic, doi=false,isbn=false,url=false, maxcitenames = 99, maxbibnames = 99]{biblatex}
\usepackage{setspace}
\usepackage[margin=1in]{geometry}

\usepackage{color,comment}
\usepackage{graphicx, floatrow, wrapfig}

\usepackage[inline]{enumitem}
\usepackage[super]{nth}
\newlist{stages}{enumerate}{1}
\setlist[stages]{label=\Roman*.}

\newtheorem{thm}{Theorem}[section]
\newtheorem{theorem}{Theorem}
\newtheorem{lem}[thm]{Lemma}
\newtheorem{prop}[thm]{Proposition}
\newtheorem{cor}[thm]{Corollary}
\newtheorem{corollary}[theorem]{Corollary}

\theoremstyle{definition}
\newtheorem{defn}[thm]{Definition}
\newtheorem{ex}[thm]{Example}
\theoremstyle{remark}
\newtheorem{rmk}[thm]{Remark}
 
\newcommand{\catname}[1]{{\normalfont\textbf{#1}}}

\newcommand{\Rex}{\catname{Rex}}

\newcommand{\Cauchy}{\catname{Cauchy}}

\newcommand{\Cocont}{\catname{Cocont}}
\newcommand{\Set}{\catname{Set}}
\newcommand{\Rep}{\catname{Rep}}
\newcommand{\Repfd}{\catname{Rep}^{\mathrm{fd}}}
\newcommand{\Cat}{\catname{Cat}}
\newcommand{\Mfld}[1]{\catname{Mfld}^{\mathrm{or}, \sqcup}_{#1}}
\newcommand{\Disc}[1]{\catname{Disc}^{\mathrm{or}, \sqcup}_{#1}}
\newcommand{\Emb}[1]{\catname{Emb}^{\mathrm{or}}_{#1}}
\newcommand{\kMod}{k\catname{Mod}}
\newcommand{\Vect}{\catname{Vect}_k}
\newcommand{\bal}{\catname{Fun}_{\mathscr{A}\operatorname{--bal}}}
\newcommand{\Diag}{\catname{Diag}_{\mathscr{D}}(\mathscr{C})}
\newcommand{\DiagCat}{\catname{Diag}_{\mathscr{D}}(\Cat_k)}
\newcommand{\Sk}{\catname{Sk}}

\newcommand{\Ribbon}{\catname{Ribbon}}

\newcommand{\LFP}[1]{\catname{LFP}_{#1}}

\newcommand{\frg}{\mathfrak{g}}
\newcommand{\qgroup}[1]{\mathcal{U}_q(#1)}

\newcommand{\skalg}{\operatorname{SkAlg}}

\newcommand{\im}{\operatorname{Im}}

\newcommand{\cauchy}{\operatorname{Cauchy}}
\newcommand{\Free}{\operatorname{Free}}

\newcommand{\Hom}{\operatorname{Hom}}
\newcommand{\Comp}{\operatorname{Comp}}

\newcommand{\Bicolim}{\operatorname{Bicolim}}

\newcommand{\Id}{\operatorname{Id}}
\newcommand{\Lie}{\operatorname{Lie}}

\newcommand{\Trace}{\operatorname{Tr}}
\newcommand{\Chg}{\operatorname{Ch}}
\newcommand{\SL}{\operatorname{SL}_2}

\newcommand{\End}{\underline{\operatorname{End}}}
\newcommand{\IHom}{\underline{\operatorname{Hom}}}

\newcommand{\Mod}[2]{#1\text{--}\mathrm{mod}_{#2}}

\newcommand{\eval}{\operatorname{eval}}
\newcommand{\act}{\operatorname{act}}
\newcommand{\ftensor}[1]{\mathbin{\widetilde{#1}}}

\usepackage[shortcuts]{extdash}

\usepackage{ulem}

\makeatletter
\newcommand{\diagramh}[3]{
\ifdefined\svg{}
\renewcommand\svg@ink@area{}%
\fi
\mathop{\raisebox{-#3}{\myincludesvggroup{scale=1}{inkscape=nolatex,inkscapeformat=png,inkscapename=#1-#2,inkscapeopt=-i #2 -j}{#1}{#2}}}%
\ifdefined\svg{}
\renewcommand\svg@ink@area{-C}
\fi
}
\makeatother
\newcommand{\smalldiagram}[1]{\diagramh{skeinrelations}{#1}{9pt}}

\usepackage[symbol*]{footmisc}
\DefineFNsymbolsTM{myfnsymbols}{
  \textdagger \ddagger
  \textdaggerdbl \ddagger
  \textasteriskcentered *
  \textsection   \mathsection
  \textbardbl    \|%
  \textparagraph \mathparagraph
}
\setfnsymbol{myfnsymbols}
\usepackage{perpage}
\MakePerPage{footnote}

\usepackage{calc}

\usepackage{etoolbox}
\patchcmd{\thmhead}{(#3)}{#3}{}{}

\usepackage{caption}

\makeatletter
\def\slashedarrowfill@#1#2#3#4#5{%
  $\m@th\thickmuskip0mu\medmuskip\thickmuskip\thinmuskip\thickmuskip
  \relax#5#1\mkern-7mu%
  \cleaders\hbox{$#5\mkern-2mu#2\mkern-2mu$}\hfill
  \mathclap{#3}\mathclap{#2}%
  \cleaders\hbox{$#5\mkern-2mu#2\mkern-2mu$}\hfill
  \mkern-7mu#4$%
}
\def\rightslashedarrowfill@{%
  \slashedarrowfill@\relbar\relbar\mapstochar\rightarrow}
\newcommand\xslashedrightarrow[2][]{%
  \ext@arrow 0055{\rightslashedarrowfill@}{#1}{#2}}
\makeatother

\begin{document}
\maketitle

\begin{abstract}
    We prove that the skein categories of Walker--Johnson-Freyd satisfy excision. 
    This allows us to conclude that skein categories are $k$-linear factorisation homology and taking the free cocompletion of skein categories recovers locally finitely presentable factorisation homology.
    An application of this is that the skein algebra of a punctured surface related to any quantum group with generic parameter gives a quantisation of the associated character variety. 
\end{abstract}

\tableofcontents

\section*{Introduction} 
\addcontentsline{toc}{section}{Introduction}
The Kauffman bracket skein algebra of the oriented surface \(\Sigma\) is the algebra of framed, oriented links in \(\Sigma \times [0,1]\) modulo the local `skein relations'
\[
\smalldiagram{g4547} = q^{-1} \smalldiagram{g4563} + q\smalldiagram{g4571} \text{ and } \smalldiagram{g4617} = -q^{2} -q^{-2}.
\]
By introducing coupons and colouring, one may define a skein algebra for any reductive algebraic group \(G\)~\footnote{The Kauffman bracket skein algebra is the case \(G=\SL\).}, or indeed any \(k\)-linear ribbon category \(\mathscr{V}\).

\begin{wrapfigure}{r}{0.33\textwidth}
  \vspace{-20pt}
  \begin{center}
    \myincludesvg{width=0.98\textwidth}{}{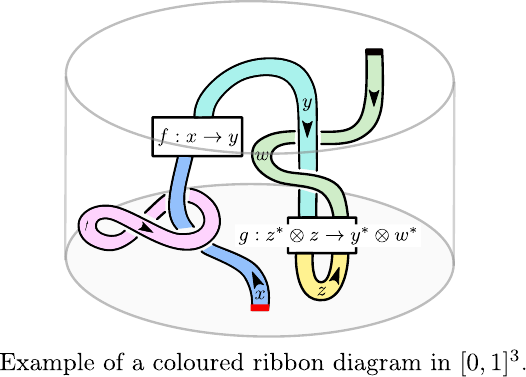}
  \end{center}
    \vspace{-20pt}
\end{wrapfigure}
A skein category extends the definition of a skein algebra further by no longer requiring the ribbon tangles to be closed. 
Skein categories were first defined by Walker \cite[70]{WalkerTQFT} and Johnson-Freyd \cite[Section~9]{JF2015}. The definition is based on the \(\mathscr{V}\)-coloured ribbon diagram category of Turaev \cites{T94Book, Tuaev97}. This \(\mathscr{V}\)-coloured ribbon diagram category is a \(k\)-linear category whose morphisms are \(\mathscr{V}\)-coloured ribbon tangles in \([0, 1]^3\); each such ribbon tangle can itself be evaluated as a morphism of \(\mathscr{V}\).
The skein category \(\Sk_{\mathscr{V}}(\Sigma)\) is then defined as the \(k\)-linear category whose 
\begin{enumerate}
    \item Objects are finite sets of framed, \(\mathscr{V}\)-coloured points in \(\Sigma\);
    \item Morphisms are \(k\)-linear combinations of \(\mathscr{V}\)-coloured ribbon tangles in \(\Sigma \times [0, 1]\) up to the equivalence that \(F \sim G\) if they are equal outside a cube and their evaluations on the cube are equal.
\end{enumerate}
For more precise definitions and some basic properties; see \cref{section:skein-categories}. The endomorphism algebra of the empty set in the skein category \(\Sk_{\mathscr{V}}(\Sigma)\) recovers the original skein algebra \(\skalg_{\mathscr{V}}(\Sigma)\) of the surface.

The main result of this paper is to prove that skein categories satisfy excision\footnote{This is \cref{thm:skeinexcision}.}:
\begin{theorem}
\label{thm:introexcision}
Let \(C\) be a \(1\)-manifold with a thickened right embedding into the boundary of the surface \(M\) and a thickened left embedding into the boundary of the surface \(N\). There is an equivalence of categories 
\[F: \Sk(M) \underset{\Sk(C \times [0, 1])}{\otimes} \Sk(N) \xrightarrow{\sim} \Sk\left(M \underset{C \times [0, 1]}{\sqcup} N \right) .\]
\end{theorem}
\noindent 
The excision of skein categories was conjectured by Johnson-Freyd \cite[Section~9]{JF2015} again based on the ideas of Walker \cite{WalkerTQFT, MW11}. There is also a result of Yetter~\cite{Yetter} which proves a similar excision result for universal braid categories in \(\Set\) and the topological parts of the proof of \cref{thm:introexcision} are based on this proof. 

Our motivation for studying skein categories came from trying to understand the relation of skein algebras to  factorisation homology. Factorisation homology takes two inputs, a topological input and an algebraic input, and produces an object which is a topological invariant of the topological input and an algebraic invariant of the algebraic input: 
\begin{itemize}
    \item The topological input is a \(n\)-dimensional manifold which may have a \(G\)-structure such as an orientation. 
    \item The algebraic input is a \(\catname{Disc}^G_n\)-algebra. This is a symmetric monoidal functor 
    \[F_A: \catname{Disc}^{G, \sqcup}_n \to \mathscr{C}^{\otimes}: \mathbb{D}^n \mapsto A\]
    from the \((\infty, 1)\)-category of \(G\)-framed discs to a symmetric monoidal \((\infty, 1)\)-category \(\mathscr{C}^{\otimes}\). 
\end{itemize}
The factorisation homology \(\int_{M}^{\mathscr{C}^{\otimes}} A\) of the manifold \(M\) with coefficients in \(A\) is obtained by extending the \(\catname{Disc}^G_n\)-algebra functor \(F_A\) to manifolds:
\[
\begin{tikzcd}
\catname{Disc}^{G, \sqcup}_n
        \ar[r, "F_A"]
        \ar[d, hook]
    & \mathscr{C}^{\otimes} \\
\catname{Mfld}^{G, \sqcup}_n
        \ar[ru, dashrightarrow, "\int^{\mathscr{C}^{\otimes}}_{\_} A"']
\end{tikzcd}
\]

Factorisation homology satisfies a generalisation of the Eilenberg--Steenrod axioms for singular homology \cite{AyalaFrancis} so may be interpreted as a generalisation of homology which is tailor-made for topological manifolds rather than general topological spaces.
For certain geometric and algebraic inputs, factorisation homology can recover other homology theories such a singular or Hochschild homology; see \cite[2]{AFPrimer} for elaboration and further examples.
In particular, factorisation homologies satisfy excision, and are characterised by this excision and a few other properties.

In \cref{section:fact-homologies}, we use the characterisation of factorisation homology in terms of excision to prove\footnote{This is \cref{thm:skeinklinear}.}
\begin{theorem}
\label{thm:introklinear}
Let \(\mathscr{V}\) be a \(k\)-linear ribbon category and \(\Cat_k^{\times}\) be the \((2,1)\)-category of small \(k\)-linear categories. The \(k\)-linear factorisation homology 
\[
\int^{\Cat_k^{\otimes}}_{\_} \mathscr{V}: \Mfld{n} \to \Cat_k^{\times}
\] 
is \(\Sk_{\mathscr{V}}(\_)\). In particular this means that for any oriented surface \(\Sigma\), there is an equivalence of categories 
\[
\Sk_{\mathscr{V}}(\Sigma) \simeq \int_{\Sigma}^{\Cat_k^{\times}} \mathscr{V}.
\]
\end{theorem}
After proving \cref{thm:introexcision} in \cref{section:skein-categories}, the remaining difficulty in proving this theorem is due to differing definitions of the relative tensor product. In \cref{section:skein-categories}, we prove that skein categories satisfy excision where the relative tensor product is Tambara's relative tensor product of \(k\)-linear categories (see \cref{defn:reltensorproduct}). In the statement of the excision of factorisation homology a different definition of relative tensor product is used: the relative tensor product is the colimit in \(\mathscr{C}^{\otimes}\) of the \(2\)-sided bar construction. In \cref{section:relative-tensor-products}, we prove that when \(\mathscr{C}^{\otimes} = \Cat_k^{\times}\), the \((2, 1)\)-category of \(k\)-linear categories, these definitions are equivalent so, in particular, skein categories satisfy the same notion of excision as factorisation homologies. 

Finally in \cref{section:character-varieties}, we apply our results to the quantisation of character varieties. Ben-Zvi, Brochier and Jordan \cite{david1} showed that one can quantise character varieties of punctured surfaces using factorisation homology. They showed that the locally finitely presentable factorisation homology \(\int^{\LFP{\mathbb{C}}}_{\Sigma_0} \Rep_q(G)\), of the punctured surface \(\Sigma_0\) with coefficients in the category \(\Rep_q(G)\) of integrable representations of the quantum group \(\qgroup{\frg}\), is equivalent to the category \( \Mod{A_{\Sigma_0}}{\Rep_q(G)}\) for a combinatorially determined algebra object \(A_{\Sigma_0}\). Furthermore, they proved that the algebra of \(\qgroup{\frg}\)-invariants \(A_{\Sigma_0}^{\qgroup{\frg}}\) is a quantisation of the character variety \(\Chg_G(\Sigma_0)\). We use \cref{thm:introklinear} to prove\footnote{This is \cref{thm:freecompskeincat}.}
\begin{theorem}
Let \(\mathscr{A}\) be an abelian, \(k\)-linear, ribbon category; \(\Free\) be the free cocompletion; \(\cauchy\) be the Cauchy completion and \(\Comp\) be the subcategory of compact projective objects.
There are equivalences of categories
\begin{align*}
\Free (\Sk_{\mathscr{A}}(\Sigma)) &\simeq \int^{\LFP{k}}_{\Sigma} \Free(\mathscr{A}) \\
\text{ and } \cauchy(\Sk_{\mathscr{A}}(\Sigma)) &\simeq \Comp\left(\int^{\LFP{k}}_{\Sigma} \Free(\mathscr{A}) \right).
\end{align*}
\end{theorem}
This result was conjectured by Ben-Zvi, Brochier and Jordan \cite[Remark~1.3]{david1}. A corollary is\footnote{This is \cref{thm:skalgalginv} and \cref{cor:skeincharacter}.} 
\begin{corollary}
Let \(G\) be a connected Lie group such that its Lie algebra \(\frg = \Lie(G)\) is semisimple, \(q \in \mathbb{C}^{\times}\) be not a root of unity, and let \(\Sigma_0\) be a punctured surface. Then there is an algebra isomorphism 
\[\skalg_{\Repfd_q(G)}(\Sigma) \cong A_{\Sigma_0}^{\qgroup{\frg}}.\]
Hence, \(\skalg_{\Repfd_q(G)}(\Sigma)\) is a quantisation of the character variety \(\Chg_G(\Sigma)\). 
\end{corollary}
This corollary was proved by direct computation by the author for the cases when \(G = \SL\) and \(\Sigma_{0}\) is the four-punctured sphere or punctured torus in \cite{Cooke18}. 
It was already well known that Kauffman bracket skein algebras are quantisations \(\SL\)-character varieties \cite{Turaev91}. That result was generalised to \(\operatorname{SL}_n\) in \cite{Sikora05}. So this corollary is a generalisation of these results to more general skein algebras.

\subsection*{Acknowledgements}
The author would like to thank David Jordan for his guidance in writing this paper. The author would also like to thank  Peter Samuelson, Matt Booth and Thomas Wright for their discussions and assistance.
The research was funded through a EPSRC studentship and ERC grant STG-637618.

\section{Skein Categories}
\label{section:skein-categories}
\subsection{Definition of a Skein Category}
\label{sec:skeincatdefn}
A skein category is a categorical analogue of a skein algebra. It was first defined by Walker and Johnson-Freyd \cites[70]{WalkerTQFT}[Section~9]{JF2015}. The definition we use follows that stated by Johnson-Freyd which is based on a generalisation to a general surfaces of the category \(\Ribbon_{\mathscr{V}}\) of coloured ribbon graphs of Reshetikhin and Turaev \cites[Chapter~1]{T94Book}[Section~4]{RT90}. We begin by defining this category \(\Ribbon_{\mathscr{V}}\) for any surface\footnote{Manifolds and surfaces are assumed throughout this paper to be finitary, smooth and oriented.}.

\begin{defn}
A \emph{ribbon graph} is constructed out of a finite number of ribbons and coupons:
\begin{enumerate}
    \item A \emph{strand} is a copy of the unit interval \([0,1]\).
    \item A \emph{ribbon} is a framed strand. The bottom base of the ribbon is \(\{\,0\,\}\) and the top base of the ribbon is \(\{\,1\,\}\). Ribbons have two possible directions: up from the bottom base to the top base \((+)\) or down from the top base to the bottom base \((-)\).  
    \item A \emph{coupon} is a copy of \([0,1]^2\). The bottom base of the coupon is \([0,1] \times \{\,0\,\}\) and the top base of the coupon is \([0,1] \times \{\,1\,\}\).
\end{enumerate}
A base of a ribbon may be attached to the base of a coupon, or to the other base of the ribbon to form an annulus; otherwise, ribbons and coupons are disjoint.
\end{defn}

\begin{figure}[H]
    \centering
    \myincludesvg{height=4cm}{}{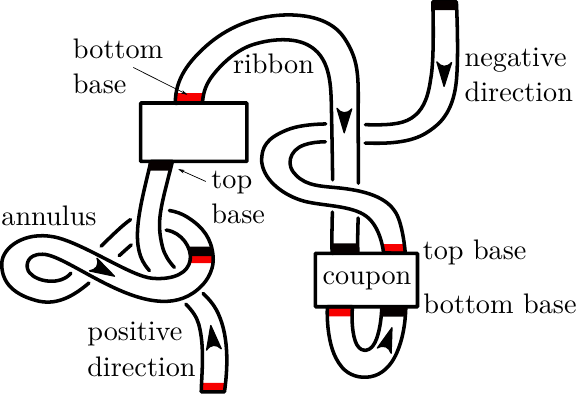} 
    \caption{A ribbon graph.}
    \label{fig:ribbongraph}
\end{figure}

\begin{defn}
Fix a strict ribbon\footnote{It is sufficient for \(\mathscr{V}\) to be a strict monoidal category with duals in this definition and for the definition of \(\Ribbon_{\mathscr{V}}(\Sigma)\). However, \(\mathscr{V}\) needs to be a ribbon category to define the ribbon functor \(\eval\) and, hence, in the definition of a skein category.  To avoid confusion we have assumed it is a ribbon category throughout.} category \(\mathscr{V}\). A ribbon graph is \emph{coloured} by \(\mathscr{V}\) as follows:
\begin{enumerate}
    \item Each ribbon is coloured with an object of \(\mathscr{V}\).
    \item For a coupon, let \(V_1, \dots, V_n\) and \(\epsilon_1, \dots, \epsilon_n\) be the colours and directions of the strands attached to the bottom base the coupon, and let \(W_1, \dots, W_m\) and \(\eta_1, \dots, \eta_m\) be the colours and directions of the strands attached to the top base the coupon---the order in which the bands are attached to bases of the coupon gives the ordering. The coupon is coloured by a morphism \(f: V_1^{\epsilon_1} \otimes \dots \otimes V_n^{\epsilon_n} \to W_1^{\eta_1} \otimes \dots \otimes W_m^{\eta_m}\) of \(\mathscr{V}\) where \(X^+ := X\) and \(X^- := X^*\) for \(X \in \mathscr{V}\). 
\end{enumerate}
\end{defn}

\begin{figure}[H]
    \centering
    \myincludesvg{height=4cm}{}{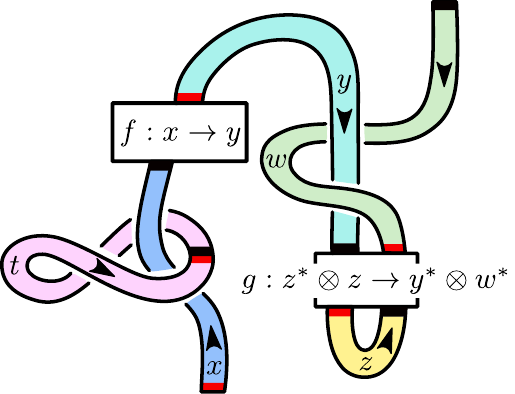}
    \caption{A coloured ribbon graph.}
    \label{fig:ribbongraphcoloured}
\end{figure}

\begin{defn}
A \emph{\(\mathscr{V}\)-coloured ribbon diagram of a surface \(\Sigma\)} is an embedding of a \(\mathscr{V}\)-coloured ribbon graph into \(\Sigma \times [0,1]\) such that unattached bases of ribbons are sent to \(\Sigma \times \{\,0, 1\,\} \) and otherwise the image lies in \(\Sigma \times (0, 1)\). The coupons must be oriented upwards.
We call its intersection with \(\Sigma \times \{\,0\,\} \) the bottom of the diagram and its intersection with \(\Sigma \times \{\,1\,\}\) the top of the diagram.
\end{defn}

\begin{defn}
Two coloured strand diagrams are \emph{isomorphic} if there is a finite sequence of isotopies from one to the other; each isotopy must be fixed except in the interior of a \(3\)-ball and preserve the ribbon graph structure i.e.\ the attachments of ribbons and coupons, directions of ribbons and colouring.
\end{defn}

\begin{defn}
Fix a strict ribbon category \(\mathscr{V}\) and a surface \(\Sigma\). The \(k\)-linear category of \(\mathscr{V}\)-coloured ribbon diagrams in \(\Sigma\) is denoted \(\Ribbon_{\mathscr{V}}(\Sigma)\): 
\begin{stages}
    \item An object of \(\Ribbon_{\mathscr{V}}(\Sigma)\) is a finite set \(\left\{\,x^{(V_1, \epsilon_1)}_1, \dots , x^{(V_n, \epsilon_n)}_n\,\right\}\) of disjoint framed points \(x_i \in \Sigma\) coloured by objects \(V_i \in \mathscr{V}\) and given directions \(\epsilon_i \in \left\{\,+, -\,\right\}\).
    \item A morphism \(F: \left\{\,x^{(V_1, \epsilon_1)}_1, \dots , x^{(V_n, \epsilon_n)}_n\,\right\} \to \left\{\,y^{(W_1, \delta_1)}_1, \dots , y^{(W_m, \delta_m)}_m\,\right\}\) is a finite linear combination \(F = \sum_i \lambda_i F_i\)  where \(\lambda_i \in k\) and \(F_i\) is a \(\mathscr{V}\)-coloured ribbon diagram such that bottom of the diagram is \(\left\{\,x^{(V_1, \epsilon_1)}_1, \dots , x^{(V_n, \epsilon_n)}_n\,\right\}\) and the top is \(\left\{\,y^{(W_1, \delta_1)}_1, \dots , y^{(W_m, \delta_m)}_m\,\right\}\)\footnote{Note that this includes the colouring, framings and directions matching.}.
    \item The identity morphism \(\Id_{\left\{\,x^{(V_1, \epsilon_1)}_1, \dots , x^{(V_n, \epsilon_n)}_n\,\right\}}\) is the ribbon diagram consisting of \(n\) ribbons which are fixed in \(\Sigma\)-coordinate and framing (up to isomorphism).
    \item The composition of morphisms \(F = \sum_i \lambda_i F_i\) and \(G = \sum_j \mu_j G_j\) is \(G \circ F = \sum_{i,j} \lambda_i \mu_j G_j \circ F_i\). The ribbon diagram
    \(G_j \circ F_i\) is formed by stacking and then retracting \(\Sigma \times [0,2]\) to \(\Sigma \times [0, 1]\). When forming \(G_j \circ F_i\) the strands of \(F_i\) attached to the top of its diagram and the strands of \(G_j\) attached to the bottom of its diagram are merged.
\end{stages}
\end{defn}

\begin{rmk}
\label{rmk:embedding}
Note that an embedding of surfaces \(p: \Sigma \to \Pi\) induces a functor \(P: \Sk(\Sigma) \to \Sk(\Pi)\) of skein categories: on the object \(\left\{\,x^{(V_i, \epsilon_i)}_i\,\right\}\), \(P\) is defined by \(P\left(x^{(V_i, \epsilon_i)} \right) = P \left(x^{(V_i, \epsilon_i)} \right)\) and, on the morphism \(F = \sum_i \lambda_i F_i\), \(P\) is defined by \(P(F_i) = \left( p \times \Id_{[0, 1]} \right) (F_i)\).
\end{rmk}
\begin{rmk}
\label{rmk:canonicalmonoidal}
When \(\Sigma = C \times [0,1]\), for some \(1\)-manifold \(C\), the \(k\)-linear category \(\Ribbon_{\mathscr{V}}(\Sigma)\) can be equipped with a monoidal structure\footnote{The monoidal unit is the empty set.} induced by the embedding
\[
I: \left( C \times [0,1] \right) \sqcup \left( C \times [0,1] \right) \xhookrightarrow{} C \times [0,1]
\]
which retracts both copies of \(C \times [0,1]\) in the second coordinate and includes them into another copy of \(C \times [0,1]\); we shall denote the retractions \(l\) and \(r\) respectively.
Furthermore, the monoidal category \(\Ribbon_{\mathscr{V}}(C \times [0,1])\) has duals with the dual of an object obtained by flipping directions: 
 \[\left\{\,x^{(V_1, \epsilon_1)}_1, \dots , x^{(V_n, \epsilon_n)}_n\,\right\}^* := \left\{\,x^{(V_1, -\epsilon_1)}_1, \dots , x^{(V_n, -\epsilon_n)}_n\,\right\}:\]
 the unit and counit are given by the cap and cup respectively. Finally, equipping \(\Ribbon_{\mathscr{V}}(C \times [0,1])\) with the braiding and twist given by crossing ribbons and twisting ribbons, as depicted in \cref{fig:twistandbraiding}, turns \(\Ribbon_{\mathscr{V}}(C \times [0,1])\) into a ribbon category. In particular, \(\Ribbon_{\mathscr{V}}([0,1]^2)\) is a ribbon category.
 \begin{figure}
     \centering
     \subfloat{\myincludesvg{height=3cm}{}{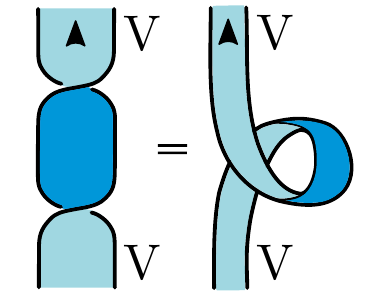}}
     \qquad
     \subfloat{\myincludesvg{height=2.9cm}{}{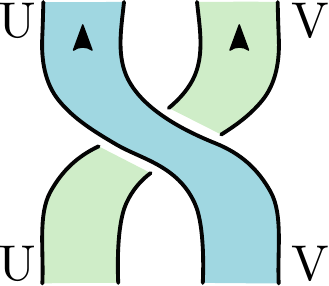}}
     \caption{The twist \(\theta_V: V \to V\) and the braiding \(B_{U, V}: U \otimes V \to V \otimes U\) of \(\Ribbon_{\mathscr{V}}(C \times [0,1])\).}
     \label{fig:twistandbraiding}
 \end{figure}
\end{rmk}

\begin{prop}[{\cite[Theorem~2.5]{T94Book}}]
\label{prop:Ribbonfunctor}
Let \(\mathscr{V}\) be a strict ribbon category. There is a full ribbon functor
\[\eval: \Ribbon_{\mathscr{V}}([0,1]^2) \to \mathscr{V}\]
which is surjective on objects.
\end{prop}

To define the skein category \(\Sk_{\mathscr{V}}(\Sigma)\), we take the ribbon diagram category \(\Ribbon_{\mathscr{V}}(\Sigma)\) and force it to locally satisfy the relations satisfied in \(\mathscr{V}\):

\begin{defn}
Let \(\Sigma\) be a surface and \(\mathscr{V}\) be a strict ribbon category. The \(k\)-linear category \(\Sk_{\mathscr{V}}(\Sigma)\) is \(\Ribbon_{\mathscr{V}}(\Sigma)\) modulo the following relation on the morphisms of \(\Ribbon_{\mathscr{V}}(\Sigma)\): the morphism \(\sum_i \lambda_i F_i \sim 0\) if there exists an orientation preserving embedding 
\[E: [0,1]^3 \xrightarrow{} \Sigma \times [0,1]\] 
satisfying the following:
\begin{enumerate}
    \item the intersection of \(F_i\) with the boundary of the cube \(E(\partial [0,1]^3)\) consists only of transverse ribbons of \(F_i\) intersecting the top and bottom edge of the cube;
    \item the \(F_i\) are equal outside of \(E([0,1]^3)\); 
    \item \(\sum_i \lambda_i \eval\left( E^{-1}(F_i \cap E([0,1]^3) \right) = 0\) (\(\eval\) is the functor from \cref{prop:Ribbonfunctor}).
\end{enumerate}
\end{defn}

Directly from the definition and \cref{prop:Ribbonfunctor} we obtain the following corollary:
\begin{cor}
\label{cor:disc}
Let \(\mathscr{V}\) be a strict ribbon category. Then there is an equivalence of ribbon categories 
\[\Sk_{\mathscr{V}}\left([0,1]^2\right) \simeq \mathscr{V}.\]
\end{cor}

\subsubsection{Diagrams}

A morphism of a skein category is a linear combination of equivalence classes of coloured ribbon diagrams in \(\Sigma \times [0,1]\). A coloured ribbon diagram \(R\) can be depicted by diagrams drawn on \(\Sigma\) in a way generalising knot diagrams as follows: Deform the ribbon diagram \(R\) so that, with the exception of ribbons attached to the top and bottom of the diagram, it lies almost parallel and very close to \(\Sigma \times \{\,1/2\,\}\). The ribbons attached to the top and bottom of the diagram can be deformed so that they are constant in the \(\Sigma\) direction except very close to \(\Sigma \times \{\,1/2\,\}\). Further deform \(R\) so that the coupons of \(R\) lie in \(\Sigma \times \{\,1/2\,\}\), so that no ribbons lie directly above or below (in the \(t\)-coordinate\footnote{Ribbon diagrams lie in \(\Sigma \times [0,1]\). In this paper we refer to \([0,1]\) as the \(t\)-coordinate or direction.}) coupons, and at most two ribbons lie above or below each other. After having deformed \(R\) in this manner we draw the projection of the ribbon diagram onto \(\Sigma \times \{\,1/2\,\}\) taking account under and over crossings, and making start and end intervals as such. The original ribbon diagram \(R\) can be recovered up to isotopy from this diagram. 

\begin{figure}[H]
    \centering
    \myincludesvg{height=4cm}{}{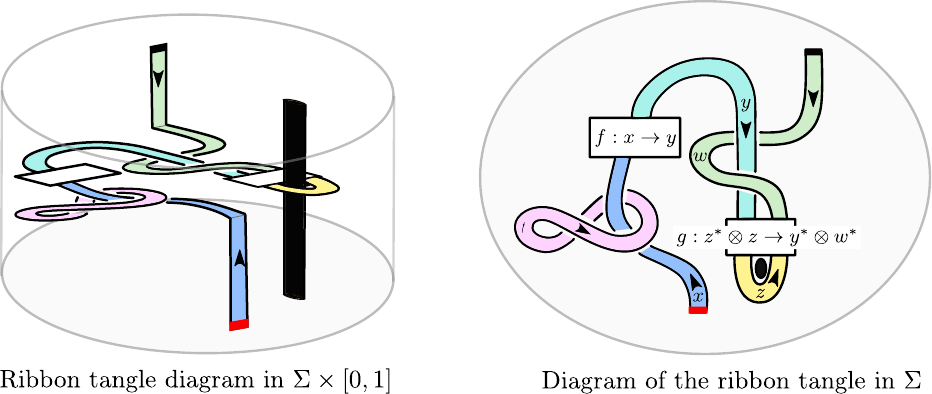}
    \label{fig:diagrams}
\end{figure}

Note that, the diagram of the composition \(S \circ R\) is obtained by placing the diagram of \(S\) on top of the diagram of \(R\) and removing the start and end points which now join up. 
\subsection{Tambara Relative Tensor Product of \(k\)-linear Categories}
\label{sec:tambarareltensorproduct}
There are two definitions of the relative tensor products used in this paper. 
\begin{enumerate}
    \item The first definition is the definition of the relative tensor product \(k\)-linear categories defined by Tambara in \cite[Section~1]{Tambara01}. It is a categorical analogue of the definition of the relative tensor product of modules. As skein categories are \(k\)-linear categories, it will be used the define the relative tensor product of skein categories which is needed to formulate the notion of excision of skein categories. 
    \item The second definition is a definition which applies to any \(\otimes\)-presentable \((\infty, 1)\)-category \(\mathscr{C}^{\otimes}\). The definition is as the colimit of the bar construction (see \cref{defn:barrelativetensorproduct}). This is the definition used to formulate the notion of excision of factorisation homologies.
\end{enumerate}
In \cref{section:relative-tensor-products}, we shall prove that when \(\mathscr{C}^{\otimes}\) is the \((2,1)\)-category of \(k\)-linear categories, \(\Cat_k^{\times}\), that these two definitions coincide. This is an important step in relating skein categories to factorisation homology (\cref{section:fact-homologies} and \cref{section:character-varieties}). However in \cref{section:skein-categories}, we only need to use Tambara's definition which we shall now define.

Tambara \cite[Section~1]{Tambara01} defines the relative tensor product \(\mathscr{M} \otimes_{\mathscr{A}} \mathscr{M}\) of the right \(\mathscr{A}\)-module \(k\)-linear category \(\mathscr{M}\) and the left \(\mathscr{A}\)-module \(k\)-linear category \(\mathscr{N}\) relative to the \(k\)-linear monoidal category \(\mathscr{A}\).

\begin{defn}
Let \(\mathscr{A}\) be a monoidal \(k\)-linear category. A \emph{left \(\mathscr{A}\)-module category} is a \(k\)-linear category \(\mathscr{M}\) equipped with a \(k\)-bilinear functor 
\[
\rhd : \mathscr{A} \otimes \mathscr{M} \to \mathscr{M}: (a, m) \mapsto a \rhd m, \]
a natural isomorphism
\[
    \beta: \_ \rhd (\_ \rhd \_) \to (\_ \otimes \_) \rhd \_ \text{ with components } \beta_{a,b,m}: a \rhd (b \rhd m) \to (a \otimes b) \rhd m 
\]
called the \emph{associator}, and a natural isomorphism
\[
    \eta : 1_{\mathscr{A}} \rhd \_ \to \_  \text{ with components } \eta_m: 1_{\mathscr{A}} \rhd m \to m
\]
called the \emph{unitor} which make the certain diagrams commute (see for example \cite[Definition~4.1.2]{CookeThesis}).
The definition for a right \(\mathscr{A}\)-module category is analogous.
\end{defn}

\begin{defn}
A bilinear functor \(F : \mathscr{M} \times \mathscr{N} \to \mathscr{C}\) is \emph{\(\mathscr{A}\)-balanced} if there is a natural isomorphism which on components is
\[\iota_{m, a, n} : F(m \lhd a, n) \to F(m, a \rhd n)\]
satisfying the commutative diagram 
\[
\begin{tikzpicture}[commutative diagrams/every diagram]
\node (P0) at (90:2.3cm) {\(((m \lhd a) \lhd b, n) \) };
\node (P1) at (90+72:2cm) {\((m \lhd a, b \rhd n)\)} ;
\node (P2) at (90+2*72:2cm) {\makebox[5ex][r]{\((m, a \rhd(b \rhd n))\)}};
\node (P3) at (90+3*72:2cm) {\makebox[5ex][l]{\((m, (a \ast b) \rhd n)\)}};
\node (P4) at (90+4*72:2cm) {\((m \lhd (a \ast b), n) \)};
\path[commutative diagrams/.cd, every arrow, every label]
(P0) edge node[swap] {\(\iota_{m \lhd a, b, n}\)} (P1)
(P1) edge node[swap] {\(\iota_{m,a,b \rhd n}\)} (P2)
(P2) edge node {\((\Id_m, \beta_{a, b, n})\)} (P3)
(P4) edge node {\(\iota_{m, a \ast b, n}\)} (P3)
(P0) edge node {\((\beta_{m,a,b},\Id_n)\)} (P4);
\end{tikzpicture}
\]
for all \(m \in \mathscr{M}\), \(a, b \in \mathscr{A}\) and \(n \in \mathscr{N}\).
\end{defn}

\begin{defn}
The natural transformation \(\alpha: F \xRightarrow{} G\) of \(\mathscr{A}\)-balanced functors \(F, G: \mathscr{M} \times \mathscr{N} \to \mathscr{C}\) is a \emph{\(\mathscr{A}\)-balanced natural transformation} if it is compatible with the balancings, i.e.\ the following diagram commutes
\[
\begin{tikzcd}[row sep=3em, column sep=3em]
F(m \lhd a, n) 
        \ar[r, "(\iota_F)_{m, a, n}"]
        \ar[d, "\alpha_{(m \lhd a, n)}"]
    & F(m, a \rhd n) 
        \ar[d, "\alpha_{(m, a \rhd n)}"]\\
G(m \lhd a, n) 
        \ar[r, "(\iota_G)_{m, a, n}"]
    & G(m, a \rhd n)
\end{tikzcd}
\]
\end{defn}

\begin{defn}
Let \(\bal(\mathscr{M}, \mathscr{N}; \mathscr{C})\) be the category with objects being the \(\mathscr{A}\)-balanced functions \(\mathscr{M} \times \mathscr{N} \to \mathscr{C}\) and morphisms being the \(\mathscr{A}\)-balanced natural transformations.
\end{defn}

\begin{defn}[{\cite[Section~1]{Tambara01}}]
Let \(\mathscr{A}\) be a \(k\)-linear monoidal category, let \(\mathscr{M}\) be a right \(\mathscr{A}\)-module \(k\)-linear category, and let \(\mathscr{N}\) be a left \(\mathscr{A}\)-module \(k\)-linear category. The \emph{relative tensor product} of \(\mathscr{M}\) and \(\mathscr{N}\) relative to \(\mathscr{A}\) is a \(k\)-linear category \(\mathscr{M} \otimes_{\mathscr{A}} \mathscr{N}\) together with a \(\mathscr{A}\)-balanced functor \(P: \mathscr{M} \times \mathscr{N} \to \mathscr{M} \otimes_{\mathscr{A}} \mathscr{N}\) such that for all \(k\)-linear categories \(\mathscr{C}\) there is an equivalence of categories 
\[ 
\Cat_k(\mathscr{M} \otimes_{\mathscr{A}} \mathscr{N}, \mathscr{C}) \simeq \bal(\mathscr{M}, \mathscr{N}; \mathscr{C})
\]
given by precomposing functors with \(P\). 
\end{defn}

Tambara proves the existence of such a relative tensor product by constructing it. 

\begin{defn}[{\cite[Section~1]{Tambara01}}]
\label{defn:reltensorproduct}
Let \(\mathscr{A}\) be a \(k\)-linear monoidal category, let \(\mathscr{M}\) be a \(k\)-linear right \(\mathscr{A}\)-module category, and let \(\mathscr{N}\) be a \(k\)-linear left \(\mathscr{A}\)-module category. The \emph{relative tensor product} \(\mathscr{M} \otimes_{\mathscr{A}} \mathscr{N}\) is the \(k\)-linear category with the following generators and relations. The objects of \(\mathscr{M} \otimes_{\mathscr{A}} \mathscr{N}\) are pairs \((m, n)\) where \(m \in \mathscr{M}\) and \(n \in \mathscr{N}\).
The morphisms of \(\mathscr{M} \otimes_{\mathscr{A}} \mathscr{N}\) are generated by morphisms \((f, g)\), where \(f : m \to m'\) is a morphism in \(\mathscr{M}\) and \(g: g \to g'\) is a morphism in \(\mathscr{N}\), and by morphisms \(\iota_{m,a,n}: (m \lhd a, n) \to (m, a \rhd n)\) and \(\iota^{-1}_{m,a,n}: (m, a \rhd n) \to (m \lhd a, n)\), where \(m \in \mathscr{M}\), \(a \in \mathscr{A}\) and \(n \in \mathscr{N}\). The morphisms satisfy the following relations:
\begin{description}
\item[Linearity] \((f + f', g) = (f, g) + (f', g)\), \((f, g + g') = (f, g) + (f, g')\) and \(a(f,g) = (af, g) = (f, ag)\);
\item[Functionality] \((f'f,g'g) = (f', g') \circ (f, g)\) and \((\Id_{m}, \Id_{n}) = \Id_{(m,n)}\);
\item[Isomorphism] \(\iota_{m,a,n} \circ \iota^{-1}_{m,a,n} = \Id_{(m, a \rhd n)}\) and \(\iota^{-1}_{m,a,n} \circ \iota_{m,a,n} = \Id_{(m \lhd a, n)}\);
\item[Naturality] \(\iota_{m', a', n'} \circ (f \lhd u, g) = (f, u \rhd g) \circ \iota_{m, a, n}\) 
\item[Pentagon]
\[
\begin{tikzpicture}[commutative diagrams/every diagram]
\node (P0) at (90:2.3cm) {\(((m \lhd a) \lhd b, n) \) };
\node (P1) at (90+72:2cm) {\((m \lhd a, b \rhd n)\)} ;
\node (P2) at (90+2*72:2cm) {\makebox[5ex][r]{\((m, a \rhd(b \rhd n))\)}};
\node (P3) at (90+3*72:2cm) {\makebox[5ex][l]{\((m, (a \ast b) \rhd n) \)}};
\node (P4) at (90+4*72:2cm) {\((m \lhd (a \ast b), n) \)};
\path[commutative diagrams/.cd, every arrow, every label]
(P0) edge node[swap] {\(\iota_{m \lhd a, b, n}\)} (P1)
(P1) edge node[swap] {\(\iota_{m,a,b \rhd n}\)} (P2)
(P2) edge node {\((\Id_m, \beta_{a, b, n})\)} (P3)
(P4) edge node {\(\iota_{m, a \ast b, n}\)} (P3)
(P0) edge node {\((\beta_{m,a,b},\Id_n)\)} (P4);
\end{tikzpicture}
\]
\item[Triangle]
\[
\begin{tikzcd}[column sep=3em, row sep=3em]
(m \lhd \Id_{\mathscr{A}}, n) 
        \ar[r, "\iota_{m, \Id_{\mathscr{A}}, n}"]
        \ar[d, "{(\theta_m, \Id_n)}"]
    & (m, \Id_{\mathscr{A}}\rhd n) 
        \ar[ld, "{(\Id_m, \eta_n)}"] \\
(m, n)
    &
\end{tikzcd}
\]
\end{description}
The \(\mathscr{A}\)-balanced bilinear functor \(P: \mathscr{M} \times \mathscr{N} \to \mathscr{M} \otimes_{\mathscr{A}} \mathscr{N}\) is defined by \(P(m,n) = (m, n)\) on objects and \(P(f,g) = (f,g)\) on morphisms. 
\end{defn}
\subsection{Module Structure and the Relative Tensor Product of Skein Categories}
Let \(C\) be a \(1\)-manifold.
We have already seen that \(\Sk(C \times [0,1])\) is a monoidal category (\cref{rmk:canonicalmonoidal}). Suppose that we have a surface \(M\) with boundary \(\partial M\). We shall now show how a suitable embedding of \(C\) into \(\partial M\) equips \(\Sk(M)\) with a \(\Sk(C \times [0,1])\)-module structure.

\begin{defn}
Let \(C\) be a \(1\)-manifold and \(M\) be a surface with boundary \(\partial M\). A \emph{thickened right embedding} of \(C\) into the boundary of \(M\) consists of
\begin{enumerate}
    \item An embedding \(\Xi : C \times (-\epsilon, 1] \xhookrightarrow{} M\) such that its restriction to \(C \times \{\,1\,\}\) gives an embedding \(\xi : C \xhookrightarrow{} \partial M\). We define the restriction \(\Phi := \Xi|_{C \times [0,1]}\) and the restriction \(\mu := \Xi|_{C \times \{\,0\,\}} \).
    \item An embedding \(E: M \xrightarrow{} M\) such that \(\im(E)\) is disjoint from \(\im(\Phi)\).
    \item An isotopy \(\lambda : M \times [0,1] \to M \) from \(\Id_M\) to \(E\) which is trivial outside of \(\im(\Xi)\). 
\end{enumerate}

A \emph{thickened left embedding} is defined similarly except \(\Xi\) is an embedding \(\Xi : C \times [0, 1 + \epsilon) \xhookrightarrow{} M\) such that its restriction to \(C \times \{\,0\,\}\) gives an embedding \(\xi : C \xhookrightarrow{} \partial M\).
\end{defn}

\begin{rmk}
\label{rmk:naturalityofiso}
Let \(F, G : M \xrightarrow{} M\) be two embeddings and let \(\sigma: M \times [0,1] \to M\) be an isotopy from \(F\) to \(G\). This isotopy traces out for any \(m \in \Sk(M)\) a ribbon tangle
\[r_{\sigma, m} : F(m) \to G(m).\]
In particular if \((\Xi, E, \lambda)\) is a thickened embedding of \(C\) into the boundary of \(M\) then the isotopy \(\lambda : M \times [0,1] \to M \) traces out for any \(m \in \Sk(M)\) a ribbon tangle \(r_{\lambda, m}: m \to E(m)\). 
We also have for any  \(a \in \Sk(C \times [0,1])\) ribbon tangles \(r_{l,a}: a \to a \ast \emptyset\) and  \(r_{r, a}: a \to \emptyset \ast a\) where \(l\) and \(r\) are the retractions used to define the monoidal structure of \(\Sk(C \times [0,1])\).
Furthermore, for any ribbon tangle \(f:m \to m'\) we have that 
\[r_{\lambda, m'} \circ f = E(F) \circ r_{\lambda, m}\]
and similarly \(r_{l,a}\) and  \(r_{r, a}\) `commute' with any ribbon tangle \(g: a \to a'\).
\end{rmk}

\begin{defn}
\label{defn:inducedmodule}
Given a thickened right embedding \((\Xi, E, \lambda)\) of \(C\) into the boundary of \(M\),
\(\Sk(M)\) is a \emph{right \(\Sk(C \times [0,1])\)-module} with \emph{action} 
\[\lhd : \Sk(M) \times \Sk(C \times [0, 1]) \to \Sk(M)\]
induced from the embedding of surfaces 
\[M \sqcup (C \times [0,1]) \to M : M \sqcup A \mapsto E(M) \sqcup \Phi(A).\]
The \emph{associator \(\beta\)} is 
defined as 
\begin{align*}
    &\beta_{m, a, b}: (m \lhd a) \lhd b \to m \lhd (a \otimes b)\\
    &\beta_{m, a, b} := r_{\lambda^{-1}, (m \lhd \emptyset) \lhd \emptyset} \sqcup \left( r_{l, \emptyset \lhd a} \circ r_{\lambda^{-1}, (\emptyset \lhd a) \lhd a} \right) \sqcup r_{r, (\emptyset \lhd \emptyset) \lhd b}
\end{align*}
and the \emph{unitor \(\eta\)} is defined as
\[\eta_{m}:= r_{\lambda, m}^{-1}: m \lhd \emptyset \to m.\]
Analogously, a thickened left embedding \((\Xi, E, \lambda)\) of \(C\) into the boundary of \(N\) defines a left \(\Sk(C \times [0,1])\)-module structure on \(\Sk(N)\).
\end{defn}

\begin{figure}[ht]
    \centering
    \myincludesvg{height=2cm}{}{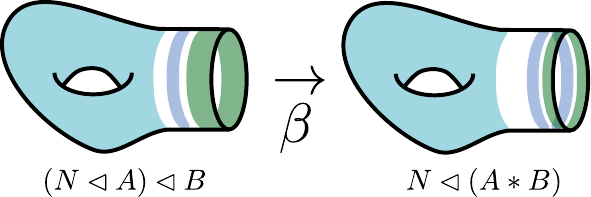}
    \label{fig:rho}
\end{figure}

As skein categories are \(k\)-linear, we may define the relative tensor product of skein categories to be their relative tensor product as \(k\)-linear categories. 

\begin{defn}
\label{defn:reltensorproductskeincats}
Let \(C\) be a \(1\)-manifold with a thickened right embedding \((\Xi_M, E_M, \lambda_M)\) into the boundary of the surface \(M\) and a thickened left embedding \((\Xi_N, E_N, \lambda_N)\) into the boundary of the surface \(N\).
By \cref{defn:inducedmodule}, \(\Sk(M)\) is a right \(\Sk(C \times [0,1])\)-module and \(\Sk(N)\) is a left \(\Sk(C \times [0,1])\)-module. 
The \emph{relative tensor product} \(\Sk(M) \otimes_{\Sk(A)} \Sk(M)\) is the relative tensor product as \(k\)-linear categories of \(\Sk(M)\) and \(\Sk(C)\) relative to \(\Sk(A)\) (See \cref{defn:reltensorproduct}).
\end{defn}

\begin{figure}[ht]
    \centering
    \myincludesvg{width=0.33\textwidth}{}{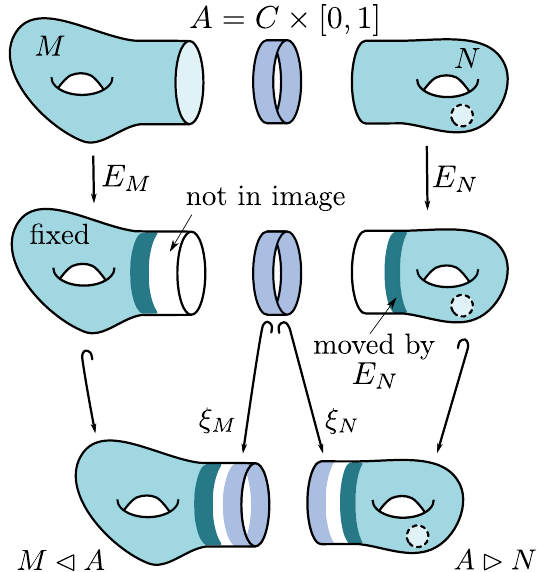}
    \label{fig:leftrightactions}
\end{figure}

\begin{rmk}
To simplify notation we shall define \(A := C \times [0, 1]\).
\end{rmk}
\subsection{Proof of the Excision of Skein Categories}
\label{sec:proofexcision}
In this subsection, we prove the following theorem:
\begin{thm}[(Excision of Skein Categories)]
\label{thm:skeinexcision} 
~\\
Let \(C\) be a \(1\)-manifold with a thickened right embedding \((\Xi_M, E_M, \lambda_M)\) into the boundary of the surface \(M\) and a thickened left embedding \((\Xi_N, E_N, \lambda_N)\) into the boundary of the surface \(N\). Let \(A := C \times [0,1]\), \(\Sk(M) \otimes_{\Sk(A)} \Sk(N)\) be the relative tensor product of skein categories (\cref{defn:reltensorproductskeincats}), and \(M \sqcup_{A} N\) be the gluing
\[
M \sqcup_{A} N := M \sqcup N \left/ \middle\{\, \xi_N(g,i) \sim \xi_N(g,1-i) \; \middle| \; g \in \bigsqcup_i\gamma_i , i \in [0,1]\, \right\}.
\]
The thickened embeddings define a \(k\)-linear functor
\[F: \Sk(M) \underset{\Sk(A)}{\otimes} \Sk(N) \xrightarrow{\sim} \Sk(M \sqcup_{A} N)\]
which is an equivalence of categories.
\end{thm}

Before proceeding to the proof of the theorem, we shall define the ribbon tangles \(\rho_{m, a, b} \in \Sk(M)\) and \(\rho_{a, b, n} \in \Sk(N)\) and prove a few identities about them. These will be needed in the proof that \(F\) is full and faithful.

\begin{defn}
Let \(m \in \Sk(M)\) and \(a, b \in \Sk(A)\) such that the points in \(a\) are disjoint from the points in \(b\).
We define the \emph{ribbon tangle \(\rho_{m, a, b} \in \Sk(M)\)} to be
\[\rho_{m, a, b} := r_{\lambda_M, m \lhd a} \sqcup \Id_{\emptyset \lhd b} : m \lhd (a \sqcup b) \to (m \lhd a) \lhd b.\]
Let \(n \in \Sk(N)\). We define the \emph{ribbon tangle \(\rho_{a, b, n} \in \Sk(N)\)} to be
\[\rho_{a, b, n} := r_{\lambda_N, b \rhd n} \sqcup \Id_{a \rhd \emptyset}: (a \sqcup b) \rhd n \to a \rhd (b \rhd n).\]
\end{defn}

\begin{lem}
For any \(m \in \Sk(M)\), \(n \in \Sk(N)\) and \(a, b \in \Sk(A)\) such that the points in \(a\) are disjoint from the points in \(b\), we have the identities:
\begin{align*}
    \rho_{m, a, b} &= \beta^{-1}_{m, a, b} \circ \left(\Id_m \lhd (r_{l, a} \sqcup r_{r, b}) \right) \\
    \rho_{a, b, n} &= \beta^{-1}_{a, b, n} \circ \left((r_{l, a} \sqcup r_{r, b}) \rhd \Id_n \right).
\end{align*}
\begin{proof}
\begin{align*}
    &\beta^{-1}_{m, a, b} \circ \left( \Id_m \lhd (r_{l, a} \sqcup r_{r, b}) \right) \\
    \quad&= \left( r_{\lambda_M, m \lhd \emptyset} 
\sqcup 
\left(r_{\lambda_M, \emptyset \lhd a} \circ r_{l^{-1}, \emptyset \lhd (a \ast \emptyset)}
\right) \sqcup 
r_{r^{-1}, \emptyset \lhd (\emptyset \ast b)} \right) \circ  \left( \Id_m \lhd (r_{l, a} \sqcup r_{r, b}) \right) \\
\quad&= r_{\lambda_M, m \lhd \emptyset} \sqcup \left( r_{\lambda_M, \emptyset \lhd a}  \circ r_{l^{-1}, \emptyset \lhd (a \ast \emptyset)} \circ r_{l, \emptyset \lhd a } \right) \sqcup \left(
r_{r^{-1}, \emptyset \lhd (\emptyset \ast b)} \circ r_{r, \emptyset \lhd b} \right) \\
\quad&= r_{\lambda_M, m \lhd a}  \sqcup \Id_{\emptyset \lhd b} \\
\quad&= \rho_{a, b, n}.
\end{align*}
The other identity is analogous.
\end{proof}
\end{lem}

\begin{lem}
For any \(m \in \Sk(M)\), \(n \in \Sk(N)\) and \(a, b \in \Sk(A)\) such that the points in \(a\) are disjoint from the points in \(b\), the following diagram commutes.
\[
\begin{tikzcd}[column sep=small]
    & (m \lhd a, b \rhd n) 
        \ar[rd, "\iota_{m, a, b \rhd n}"]
    &
    \\
((m \lhd a) \lhd b, n)
        \ar[ur, "\iota_{m \lhd a, b, n}"]
    &
    & (m, a \rhd (b \rhd n))
    \\
    &
    &
    \\
(m \lhd (a \sqcup b), n) 
        \ar[rr, "\iota_{m, a \sqcup b, n}"]
        \ar[uu, "{(\rho_{m, a, b}, \Id_n)}"]
    &
    & (m, (a \sqcup b) \rhd n)
        \ar[uu, "{(\Id_m, \rho_{a, b, n})}"']
\end{tikzcd}
\]
We shall refer to this diagram as the pentagon.
\begin{proof}
\[
\begin{tikzcd}[column sep=small]
    & (m \lhd a, b \rhd n) 
        \ar[rd, "\iota_{m, a, b \rhd n}"]
    &
    \\
((m \lhd a) \lhd b, n)
        \ar[ur, "\iota_{m \lhd a, b, n}"]
    & \text{pentagon of } \beta
    & (m, a \rhd (b \rhd n))
    \\
(m \lhd (a \ast b), n)
        \ar[u, "{(\beta^{-1}_{m, a, b}, \Id_n)}"]
        \ar[rr, "\iota_{m, a \ast b, n}"]
    &
    & (m, (a \ast b) \rhd n)
        \ar[u, "{(\Id_m, \beta^{-1}_{a, b, n})}"']
    \\
    & \text{naturality of }\iota
    &
    \\
(m \lhd (a \sqcup b), n) 
        \ar[rr, "\iota_{m, a \sqcup b, n}"]
        \ar[uu, "{(\Id_m \lhd (r_{l, a} \sqcup r_{r, b} ), \Id_n)}"]
    &
    & (m, (a \sqcup b) \rhd n)
        \ar[uu, "{(\Id_m, (r_{l, a}, r_{r, b}) \rhd \Id_n)}"']
\end{tikzcd}
\]
\end{proof}
\end{lem}

\begin{lem}
Let \(f: m \to m'\) be a morphism in \(\Sk(M)\), \(g: a \to a'\) be a morphism in \(\Sk(M)\) which is disjoint from \(\Id_b\) and \(h: b \to b'\) be a morphism in \(\Sk(M)\) which is disjoint from \(\Id_a\).
The following diagrams commute:
\[
\begin{tikzcd}[column sep=5em]
    m \lhd (a \sqcup b) 
            \ar[r, "f \lhd (\Id_a \sqcup h)"]
            \ar[d, "\rho_{m, a, b}"]
        & m' \lhd (a \sqcup b') 
            \ar[d, "\rho_{m', a, b'}"]
    \\
    (m \lhd a) \lhd b
            \ar[r, "(f \lhd \Id_{a}) \lhd h"]
        & (m' \lhd a) \lhd b'
\end{tikzcd}
\]
\[
\begin{tikzcd}[column sep=5em]
    m \lhd (a \sqcup b) 
            \ar[r, "f \lhd (g \sqcup \Id_b)"]
            \ar[d, "\rho_{m, a, b}"]
        & m' \lhd (a' \sqcup b) 
            \ar[d, "\rho_{m', a, b'}"]
    \\
    (m \lhd a) \lhd b
            \ar[r, "(f \lhd g) \lhd \Id_b"]
        & (m' \lhd a') \lhd b
\end{tikzcd}
\]
We have a similar result for the \(\rho\) in \(\Sk(N)\). We shall refer to this as the naturality of \(\rho\).
\begin{proof}
This follows from the similar naturality of \(r_{\lambda_M}\) and \(r_{\lambda_N}\).
\end{proof}
\end{lem}

We now proceed to the proof of excision.

\begin{proof}[Proof of \cref{thm:skeinexcision}]
We shall first define 
 \[
F: \Sk(M) \underset{\Sk(A)}{\otimes} \Sk(N) \to \Sk(M \sqcup_{A} N)
\]
and show this definition is well-defined, and then show that \(F\) is full, faithful and essentially surjective. 
\subsubsection*{Definition of \texorpdfstring{\(F\)}{F}}
\begin{description}
    \item[Objects:] Let \((m, n)\) be an object of \(\Sk(M) \otimes_{\Sk(A)} \Sk(N)\), so \(m\) is a finite set of disjoint framed directed coloured points in \(M\) and \(n\) is a finite set of disjoint framed directed coloured points in \(N\). We define
    \[F(m, n) :=  E_M(m) \sqcup E_N(n)\]
    which is a finite set of disjoint framed directed coloured points in \(M \sqcup_A N\), and thus is a object of \(\Sk(M \sqcup_A N)\).
    \item[Morphisms:] By the definition of the relative tensor product (\cref{defn:reltensorproduct}), the morphisms of  \(\Sk(M) \otimes_{\Sk(A)} \Sk(N)\) are generated by the morphisms
    \begin{enumerate}
        \item  \((f, g): (m, n) \to (m', n')\), where \(f \in \Hom_{\Sk(M)}(m, m')\) and \(g \in \Hom_{\Sk(N)}(n, n')\),
        \item \(\iota_{m, a, n} :(m \lhd a, n) \to (m, a \rhd n)\) for \((m, a, n) \in \Sk(M) \times \Sk(A) \times \Sk(N)\), and 
        \item \(\iota^{-1}_{m, a, n} :(m, a \rhd n) \to (m \lhd a, n)\) for \((m, a, n) \in \Sk(M) \times \Sk(A) \times \Sk(N)\),
    \end{enumerate}
    so to define \(F\) it suffices to define \(F\) for these morphisms: 
    \begin{enumerate}
        \item \(F(f,g) := E_M(f) \sqcup E_N(g) \in \Hom_{\Sk(M \sqcup_A N)}(E_M(m) \sqcup E_N(n), E_M(m') \sqcup E_N(n'))\) where \(E\) is the functor of categories induced by the embedding \(E\).
        \item \sloppy \(F(\iota_{m, a, n}) :=  r_{\lambda_M, E^2_M(m)}^{-1} \sqcup \left( r_{\lambda_N, a} \circ r^{-1}_{\lambda_M, E_M(a)} \right) \sqcup r_{\lambda_N, E(n)} \in \Hom_{\Sk(M \sqcup_A N)}(E^2_M(m)  \sqcup E_M(a) \sqcup E_N(n), E_M(m) \sqcup E_N(a) \sqcup E_N^2(n))\) 
        \item \(F(\iota^{-1}_{m, a, n}) :=  r_{\lambda_M, E_M(m)} \sqcup \left( r_{\lambda_M, a} \circ r^{-1}_{\lambda_N, E_N(a)} \right) \sqcup r^{-1}_{\lambda_N, E_N(n)} \in \Hom_{\Sk(M \sqcup_A N)}(E_M(m) \sqcup E_N(a) \sqcup E_N^2(n), E_M^2(m) \sqcup E_M(a) \sqcup E_N(n))\) 
    \end{enumerate}
\end{description}
\begin{figure}[H]
\floatbox[{\capbeside\thisfloatsetup{capbesideposition={right,center},capbesidewidth=5.5cm}}]{figure}[\FBwidth]
{\caption{This embedding of surfaces induces a functor \(\Sk(M) \times \Sk(N) \to \Sk(M \sqcup_A N)\) 
of their skein categories. 
The functor \(F\) on \(P\left( \Sk(M) \times \Sk(N) \right)\) is given by this functor: that is on objects and on morphisms of the form \((f, g)\).
}}
{\myincludesvg{width=6cm}{}{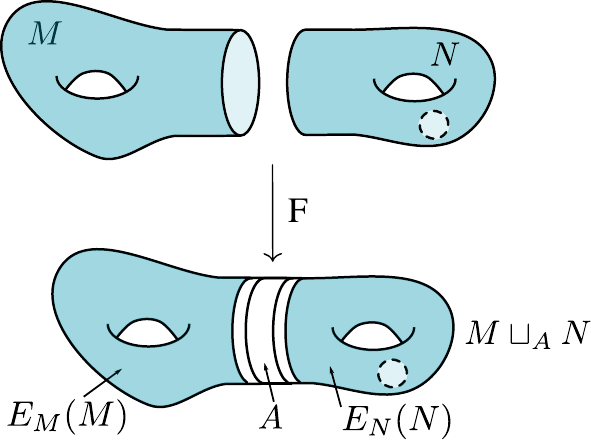}}
\end{figure}

\begin{figure}[H]
\floatbox[{\capbeside\thisfloatsetup{capbesideposition={right,center},capbesidewidth=4.5cm}}]{figure}[\FBwidth]
{\caption{The functor \(F\) on the natural isomorphism \(\iota\) gives a ribbon which has stands which cross the middle section from \(F(\emptyset \lhd a, \emptyset)\) to \(F(\emptyset,  a \rhd \emptyset)\) (coloured red). Elsewhere applying \(F(\iota_{m, a, n})\) only moves points a little.
}}
{\myincludesvg{width=7cm}{}{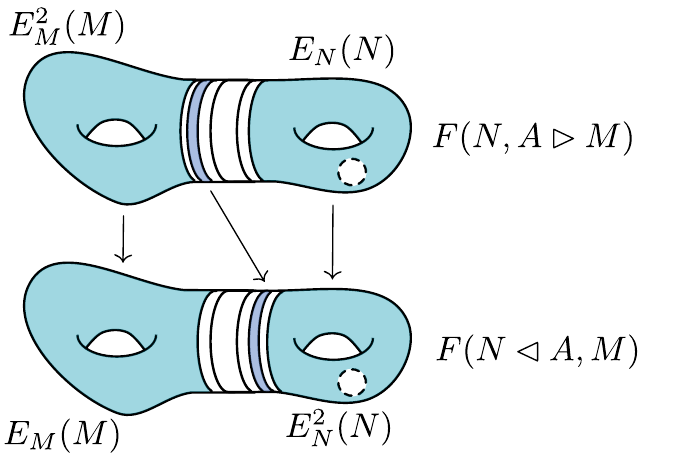}}
\label{fig:iota}
\end{figure}

In order to show that \(F\) is well-defined we must show \(F(morphism)\) still satisfies the relations in \cref{defn:reltensorproduct}. This is a sequence of straight forward calculations:
\begin{description}[style=unboxed,leftmargin=0cm]
\item[Linearity] Follows automatically as we have defined \(F\) to be \(k\)-linear.
\item[Functionality]  Follows from the functionality of the functors \(E_M\) and \(E_N\):
\begin{align*}
    F((f', g') \circ (f , g)) &= E_M(f' \circ f) \sqcup E_N(g' \circ g) \\
    &= \left( E_M(f') \circ E_M(f) \right) \sqcup \left( E_N(g') \circ E_N(g) \right) \\
    &= F((f' \circ f, g' \circ g)) \\
    F(\Id_{m, n}) &= (E_M(\Id_{m}), E_N(\Id_{n})) = (\Id_{E_M(m)}, \Id_{E_N(n)}) = \Id_{F(m, n)}
\end{align*}
\item[Isomorphism]  Follows directly from the definitions:
\begin{align*}
F(\iota_{m, a, n}) \circ F(\iota^{-1}_{m, a, n}) 
&= \left( r_{\lambda_M, E_M^2(m)}^{-1} 
\sqcup \left( r_{\lambda_N, a} \circ r^{-1}_{\lambda_M, E_M(a)} \right) 
\sqcup r_{\lambda_N, E(n)} \right) \\
&\quad \circ \left(  r_{\lambda_M, E_M(m)} 
\sqcup \left( r_{\lambda_M, a} \circ r^{-1}_{\lambda_N, E_N(a)} \right) 
\sqcup r^{-1}_{\lambda_N, E^2_N(n)} \right) \\
&= \Id_{E_M(m)} \sqcup \Id_{E_N(a)} \sqcup \Id_{E_N^2(n)} \\
&= \Id_{E(m) \sqcup E(a \rhd n)}
\end{align*}
and similarly for \( F(\iota^{-1}_{m, a, n}) \circ F(\iota_{m, a, n})\).
\item[Naturality] This follows from \cref{rmk:naturalityofiso}:
\begin{align*}
   F(\iota_{m', a', n'}) &\circ F(f \lhd g, h)\\
    &= \left( r_{\lambda_M, E^2_M(m')}^{-1}  \sqcup \left( r_{\lambda_N, a'} \circ r^{-1}_{\lambda_M, E_M(a')} \right) \sqcup r_{\lambda_N, E_N(n')} \right) \\
    &\quad\circ \left( E_M^2(f) \sqcup E_M(g) \sqcup E_N(h) \right) \\
   &= \left( r_{\lambda_M, E^2_M(m')}^{-1} \circ E^2_M(f) \right) 
    \sqcup \left( r_{\lambda_N, a'} \circ r^{-1}_{\lambda_M, E_M(a')} \circ E_M(g) \right) \\
    &\quad\sqcup \left(r_{\lambda_N, E(n')} \circ E(h) \right) \\
    &=  \left( E_M(f) \circ r_{\lambda_M, E^2_M(m)}^{-1} \right) 
    \sqcup \left(E_N(g) \circ r_{\lambda_N, a} \circ r^{-1}_{\lambda_M, E_M(a)} \right) \\
    &\quad\sqcup \left(E^2(h) \circ r_{\lambda_N, E(n)} \right) \\
   &= F(f, g \rhd h) \circ F(\iota_{m, a, n})
\end{align*}
\item[Triangle]
Follows from the definitions:
\begin{align*}
    F(\Id_m, \eta_n) \circ F(\iota_{m, \emptyset, n}) 
    &= \left(\Id_{E(m)} \sqcup r^{-1}_{\lambda_N, E_N(n)} \right) 
    \circ \left(r^{-1}_{\lambda_M, E^2_M(m)} \sqcup r_{\lambda_N, E_N(n)} \right) \\
    &= r^{-1}_{\lambda_M, E^2(m)} \sqcup \Id_{E_N(n)} \\
    &= F(\theta_{m \lhd \emptyset} ,\Id_n)
\end{align*}
\item[Pentagon]
As
\begin{align*}
\beta_{a, b, n} &= r_{\lambda_N^{-1}, \emptyset \rhd (\emptyset \rhd n) } \sqcup \left( r_{r, b \rhd \emptyset} \circ r_{\lambda_N^{-1}, \emptyset \rhd (b \rhd \emptyset)} \right) \sqcup r_{l, a \rhd \emptyset} \\
\beta^{-1}_{m, a, b} &=  r_{\lambda_M, m \lhd \emptyset} 
\sqcup 
\left(r_{\lambda_M, \emptyset \lhd a} \circ r_{l^{-1}, \emptyset \lhd (a \ast \emptyset)}
\right) 
\sqcup 
r_{r^{-1}, \emptyset \lhd (\emptyset \ast b)}
\end{align*}
we have that 
\begin{align*}
F((\beta^{-1}_{m, a, b}, \Id_n) &=  
r_{\lambda_M, E_M^2(m)} 
\sqcup 
\left(r_{\lambda_M, E_M (a)} \circ r_{l^{-1}, E_M (a \ast \emptyset)}
\right) 
\sqcup 
r_{r^{-1}, E_M(\emptyset \ast b)}
\sqcup
\Id_{E_N(n)} \\
F(\Id_m, \beta_{a, b, n}) &= 
\Id_{E_M(m)}
\sqcup
r_{\lambda_N^{-1}, E_N^3(n) } 
\sqcup 
\left( r_{r, E_N(b)} \circ r_{\lambda_N^{-1}, E_N^2(b)} \right) 
\sqcup
r_{l, E_N(a)}
\end{align*}
So,
\begin{align*}
   &F(\Id_m, \beta_{a, b, n}) \circ F(\iota_{m, a, b \rhd n}) \circ F(\iota_{m \lhd a, b, n}) \circ F(\beta^{-1}_{m, a, b}, \Id_n) \\
   &\quad=
    \left( \Id_{E_M(m)}
        \sqcup r_{\lambda_N^{-1}, E_N^3(n) } 
        \sqcup \left( r_{r, E_N(b)} \circ r_{\lambda_N^{-1}, E_N^2(b)} \right) 
        \sqcup r_{l, E_N(a)} \right) \\
&\quad \quad \circ 
    \left( r_{\lambda_M^{-1}, E^2_M(m)} 
        \sqcup \left( r_{\lambda_N, a} \circ r_{\lambda_M^{-1}, E_M(a)} \right) 
        \sqcup r_{\lambda_N, E(b \rhd n)} \right)\\
&\quad \quad \circ  
    \left( r_{\lambda_M^{-1}, E^2_M(m \lhd a)} 
        \sqcup \left( r_{\lambda_N, b} \circ r_{\lambda_M^{-1}, E_M(b)} \right) 
        \sqcup r_{\lambda_N, E(n)} \right) \\
&\quad \quad \circ  
    \left( r_{\lambda_M, E_M^2(m)} 
        \sqcup \left(r_{\lambda_M, E_M (a)} \circ r_{l^{-1}, E_M (a \ast \emptyset)} \right) 
        \sqcup r_{r^{-1}, E_M(\emptyset \ast b)}
        \sqcup \Id_{E_N(n)} \right) \\
&\quad =  
    \left( \Id_{E_M(m)} \circ r_{\lambda_M^{-1}, E^2_M(m)} \circ r_{\lambda_M^{-1}, E^3_M(m)} \circ r_{\lambda_M, E_M^2(m)} \right) \\
&\quad \quad \sqcup \left(r_{l, E_N(a)} \circ r_{\lambda_N, a} \circ r_{\lambda_M^{-1}, E_M(a)} \circ r_{\lambda_M^{-1}, E^2_M(a)}  \circ r_{\lambda_M, E_M (a)} \circ r_{l^{-1}, E_M (a \ast \emptyset)} \right) \\
&\quad \quad \sqcup \left( r_{r, E_N(b)} \circ r_{\lambda_N^{-1}, E_N^2(b)} \circ r_{\lambda_N, E(b)} \circ  r_{\lambda_N, b} \circ r_{\lambda_M^{-1}, E_M(b)} \circ r_{r^{-1}, E_M(\emptyset \ast b)} \right) \\
&\quad \quad \sqcup \left(  r_{\lambda_N^{-1}, E_N^3(n) }  \circ r_{\lambda_N, E^2(n)} \circ r_{\lambda_N, E(n)} \circ \Id_{E_N(n)} \right) \\
&\quad =  
    \left( r_{\lambda_M^{-1}, E^2_M(m)} \right) \sqcup \left(r_{l, E_N(a)} \circ r_{\lambda_N, a} \circ r_{\lambda_M^{-1}, E_M(a)} \circ r_{l^{-1}, E_M (a \ast \emptyset)} \right) \\
&\quad \quad \sqcup \left( r_{r, E_N(b)} \circ  r_{\lambda_N, b} \circ r_{\lambda_M^{-1}, E_M(b)} \circ r_{r^{-1}, E_M(\emptyset \ast b)} \right) \sqcup \left( r_{\lambda_N, E(n)}  \right) \\
&\quad =  
    \left( r_{\lambda_M^{-1}, E^2_M(m)} \right) \sqcup \left( r_{\lambda_N, a \ast \emptyset} \circ r_{\lambda_M^{-1}, E_M(a \ast \emptyset)} \right)  \sqcup \left( r_{\lambda_N, \emptyset \ast b} \circ r_{\lambda_M^{-1}, E_M(\emptyset \ast b)} \right) \\
    &\quad\quad\sqcup \left( r_{\lambda_N, E(n)}  \right) \\
&\quad = F(\iota_{m, a \ast b, n})
\end{align*}
\end{description}
\begin{rmk}
These identities have straightforward interpretations topologically, for example the pentagon identity holds as one can straighten strands. 
\end{rmk}

\subsubsection*{\texorpdfstring{\(F\)}{F} is essentially surjective}
Any point in \(E_M(M) \sqcup E_N(N) \subset M \sqcup_A N\) is in the image of \(F\). If point \(x^{(V, \epsilon)}\) is not in this region then there is a ribbon which translates \(x^{(V, \epsilon)}\) across the middle region to a point \(\tilde{x}^{(V, \epsilon)}\) which is in this region. Hence, every point in \(M \sqcup_A N\) is isomorphic to an point in the image of \(F\), and \(F\) is essentially surjective. 

\subsubsection*{\texorpdfstring{\(F\)}{F} is full}
Let \((m_1, n_1), (m_2, n_2)\)  be any objects in \(\Sk(M) \otimes_{\Sk(A)} \Sk(N)\) and let 
\[ 
[\overline{u}] \in \Hom_{\Sk(M \sqcup_A N)} \left( F(m_1, n_1), F(m_2, n_2) \right),
\] 
so \([\overline{u}]\) is the equivalence class of a ribbon diagram
\[
\overline{u}: E_M(m_1) \sqcup E_N(n_1) \to E_M(m_2) \sqcup E_N(n_2).
\]
In order to show \(F\) is full, we must show there is a morphism 
\[ w \in \Hom_{\Sk(M) \otimes_{\Sk(A)} \Sk(N)} \left( (m_1, n_1), (m_2, n_2) \right)
\] 
such that \(F(w) = u\) for some \(u\) equivalent to \(\overline{u}\).

We shall call \(\im( \Xi_M \cup \Xi_N) \times [0,1]\), the \emph{middle region}.  
Up to isotopy fixed outside this middle region, we may assume that \(\overline{u}\) intersects \(\im(\mu_M) \times [0,1]\) in a finite number of transverse strands. 
Let \(t_i \in [0,1]\) be the levels when \(u_{t_i}\) intersects \(\im(\mu_M)\). 
By an isotopy in the \(t\)-coordinate which moves coupons, twists, minima, maxima, and strands not lying in \(F(M, N)\)\footnote{Being able to do this relies on the ribbon diagram \(u\) not starting or ending in the middle region.}, we may assume that \(u_{t_i}\) consists entirely of framed points in \(F(M, N)\).
Up to isotopy fixed in the middle region, we can further assume \(u_{t_i}\) contains framed points entirely in \(F(M \lhd A, A \rhd N)\). This means that \(u_{t_i} = \left(m \lhd \left(a \sqcup \overline{b} \right), \left(\overline{c} \sqcup d \right) \rhd n \right)\) where only \(\overline{b}\) and \(\overline{c}\) intersect \(\im(\mu_M) \sqcup \im(\mu_N)\).
We reparametrise further so that for some small \(\epsilon_i>0\), 
\(u_{[t_i-\epsilon_i, t_i + \epsilon]} = \Id_{m \lhd a, d \rhd n} \sqcup v_{[t_i-\epsilon_i, t_i + \epsilon]}\) where \(v_{[t-\epsilon, t+\epsilon]}: E_M(b) \sqcup E_N(c) \to E_M(b') \sqcup E_N(c')\): in other words \(u_{[t_i-\epsilon_i, t_i + \epsilon]}\) consists of identity strands and a ribbon tangle which straddles the middle region. 

\begin{figure}[H]
    \centering
    \myincludesvg{height=3cm}{}{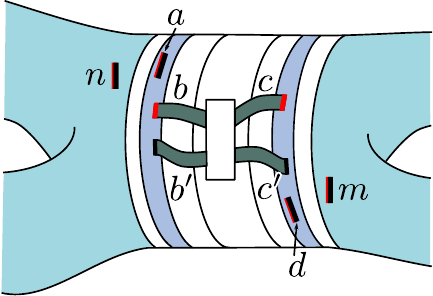}
    \caption{An example of \(u_{[t_i-\epsilon_i, t_i + \epsilon]}\). In general \(a, b, b', \dots\) are not single framed points, but finite sets of framed points, and the coupon depicted could be any ribbon diagram in this square with the same inputs and outputs.}
    \label{fig:mapv}
\end{figure}

We now have a ribbon diagram \(u\) equivalent to \(\overline{u}\) with a decomposition
\[u = u_{[1, t_N + \epsilon_N]} 
    \circ u_{[t_N - \epsilon_N, t_N + \epsilon_N]} 
    \circ u_{[t_{N-1} + \epsilon_{N-1}, t_N - \epsilon_N]} 
    \circ \dots \circ u_{[t_1 - \epsilon_1, t_1 + \epsilon_1]} 
    \circ u_{[0, t_1 - \epsilon_1]}\]
such that 
\(u_{[t_i-\epsilon_i, t_i + \epsilon_i]} = \Id_{m_i \lhd a_i, d_i \rhd n_i} \sqcup v_{[t_i-\epsilon_i, t_i + \epsilon_i]}\) and the other morphisms in the decomposition lie in \(F(M, N) \times [0, 1]\). If a morphism lies in \(F(M, N) \times [0, 1]\) then it is of the form \(f \sqcup g\) for \(f \in F(M) \times [0, 1]\) and \(g \in F(N) \times [0, 1]\). In which case \(F(E^{-1}_{M}(f), E^{-1}_N(g)) = (f,g)\).
So it remains to consider the ribbon tangle
\begin{align*}
    &u_{[t-\epsilon, t+\epsilon]} : F(m \lhd (a \sqcup b),  (c \sqcup d) \rhd n) \to F(m \lhd (a \sqcup b'),  (c' \sqcup d) \rhd n) \\
    &u_{[t-\epsilon, t+\epsilon]} = v_{[t-\epsilon, t+\epsilon]} \sqcup \Id_{E^2_M(m) \sqcup E_M(a) \sqcup E_N(d) \sqcup E^2_N(n)}
\end{align*}
where \(v_{[t-\epsilon, t+\epsilon]}: E_M(b) \sqcup E_N(c) \to E_M(b') \sqcup E_N(c)\). As the middle region is homeomorphic to \(C \times [0, 1]\) and \([0, 1]\) is topologically trivial, there exists a ribbon tangle \(\overline{v}: b \rhd c \to b' \rhd c'\) in \(\Sk(M)\) such that 
\begin{align*}
v_{[t-\epsilon, t+\epsilon]} 
&= \left( \left( r_{\lambda_M, b'} \circ r^{-1}_{\lambda_N, E_N(b')} \right) \sqcup r_{\lambda_N, E_N(c')}\right) 
\circ E_N(\overline{v}) \\ 
&\quad\circ \left( \left( r_{\lambda_N, b} \circ r^{-1}_{\lambda_M, E_M(b)} \right) \sqcup r_{\lambda^{-1}_N, E^2_N(c)} \right).
\end{align*}

\begin{figure}[H]
\floatbox%
[{\capbeside\thisfloatsetup{capbesideposition={right,center},capbesidewidth=8cm}}]%
{figure}[\FBwidth]%
{%
\captionsetup{singlelinecheck=off}
\protect{%
\caption%
[The top figure is an example of \(v_{[t-\epsilon, t+\epsilon]}\). The bottom figure is isotopic and depicts the decomposition of \(v_{[t-\epsilon, t+\epsilon]}\).]%
{%
The top figure is an example of \(v_{[t-\epsilon, t+\epsilon]}\). The bottom figure is isotopic and depicts the decomposition of \(v_{[t-\epsilon, t+\epsilon]}\): 
\begin{enumerate}
    \item The yellow ribbons are 
    \[\left( \left( r_{\lambda_N, b} \circ r^{-1}_{\lambda_M, E_M(b)} \right) \sqcup r_{\lambda^{-1}_N, E^2_N(c)} \right);\]
    \item The distorted copy of \(v_{[t-\epsilon, t+\epsilon]}\) is \(\overline{v}\);
    \item The blue ribbons are 
    \[
    \left( \left( r_{\lambda_M, b'} \circ r^{-1}_{\lambda_N, E_N(b')} \right) \sqcup r_{\lambda_N, E_N(c')}\right).
    \]
\end{enumerate}
}
}%
}
{\myincludesvg{width=4cm}{}{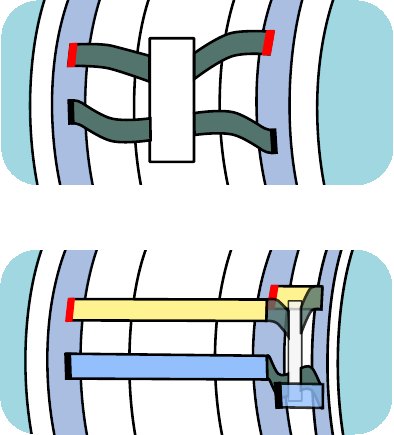}}
\label{fig:vshift}
\end{figure}
We denote by \(w_{[t-\epsilon, t+\epsilon]}\) the following morphism in \(\Sk(M) \underset{\Sk(A)}{\otimes} \Sk(N)\):
\[
\begin{tikzcd}[cramped]
(m \lhd (a \sqcup b), (c \sqcup d) \rhd n) 
        \ar[d, "{(\rho_{m, a, b}, \Id)}"] \\
((m \lhd a) \lhd b, (c \sqcup d) \rhd n)
        \ar[d, "\iota_{m \lhd a, b, (c \sqcup d) \rhd n}"] \\
(m \lhd a, b \rhd ((c \sqcup d) \rhd n))
        \ar[d, "{(\Id_{m \lhd a}, \overline{v} \sqcup \Id_{\emptyset \rhd (d \rhd n)})}"] \\
(m \lhd a, b' \rhd ((c' \sqcup d) \rhd n))
        \ar[d, "\iota^{-1}_{m \lhd a, b', (c' \sqcup d) \rhd n}"] \\
((m \lhd a) \rhd b' ,(c' \sqcup d) \rhd n)
        \ar[d, "{(\rho^{-1}_{m, a, b'}, \Id)}"] \\
(m \lhd (a \sqcup b') ,(c' \sqcup d) \rhd n)
\end{tikzcd}
\]
We shall sometimes denote \(\hat{v} := \overline{v} \sqcup \Id_{\emptyset \rhd (d \rhd n)}\).
We claim that \(F(w_{[t-\epsilon, t+\epsilon]})= v_{[t-\epsilon, t+\epsilon]}\). By the functorality of \(F\) and the definition of \(F\) on the various components, we have that \(F(w_{[t-\epsilon, t+\epsilon]})\) is
\[
\begin{tikzcd}[cramped]
F(m \lhd (a \sqcup b), (c \sqcup d) \rhd n) 
        \ar[d, "{r_{\lambda_M, E_M(m \lhd a)} \sqcup \Id_{E_M(b)}  \sqcup \Id_{E_N(c)} \sqcup \Id_{E_N(d \rhd n)}}"] \\
F((m \lhd a) \lhd b, (c \sqcup d) \rhd n)
        \ar[d, "{r_{\lambda^{-1}_M, E_M(m \lhd a)} \sqcup \left( r_{\lambda_N, b} \circ r^{-1}_{\lambda_M, E_M(b)} \right) \sqcup r_{\lambda_N, E_N(c)} \sqcup r_{\lambda_N, E_N( d \rhd n)}}"] \\
F(m \lhd a, b \rhd ((c \sqcup d) \rhd n))
        \ar[d, "{\left(\Id_{E_M(m \lhd a)}, E_N(\overline{v}) \sqcup \Id_{E^2_N(d \rhd n)}\right)}"] \\
F(m \lhd a, b' \rhd ((c' \sqcup d) \rhd n))
        \ar[d, "r_{\lambda_M, E_M(m \lhd a)} \sqcup \left( r_{\lambda_M, b'} \circ r_{\lambda_N^{-1}, E_N(b')} \right) \sqcup r^{-1}_{\lambda_N, E_N(c')} \sqcup r^{-1}_{\lambda_N, E_N(d \rhd n)}"] \\
F((m \lhd a) \rhd b' ,(c' \sqcup d) \rhd n)
        \ar[d, "{r_{\lambda^{-1}_M, E^2_M(m \lhd a)} \sqcup \Id_{E_M(b)} \sqcup \Id_{E_N(c')} \sqcup \Id_{E_N(d \rhd n)}}"] \\
F(m \lhd (a \sqcup b') ,(c' \sqcup d) \rhd n)
\end{tikzcd}
\]
So it decomposes into three components: 
\begin{align*}
F(w_{[t-\epsilon, t+\epsilon]}) 
&= \Id_{E_M(m \lhd a)} \sqcup  \Id_{E_N(d \rhd n)} \\
&\quad\sqcup \left( \left( r_{\lambda_M, b'} \circ r^{-1}_{\lambda_N, E_N(b')} \right) \sqcup r_{\lambda_N, E_N(c')}\right) 
\circ E_N(\overline{v}) \\
&\quad\circ \left( \left( r_{\lambda_N, b} \circ r^{-1}_{\lambda_M, E_M(b)} \right) \sqcup r_{\lambda^{-1}_N, E^2_N(c')} \right) \\
&= \Id_{E_M(m \lhd a)} \sqcup  \Id_{E_N(d \rhd n)} \sqcup v \\
&= u_{[t-\epsilon, t+\epsilon]}.
\end{align*}
and we are done. 

\subsubsection*{\texorpdfstring{\(F\)}{F} is faithful}
In the previous section we have shown that for any ribbon tangle \(u\) there is a morphism \(w\) such that \(F(w) = u\). 
We shall now show that this defines a well defined inverse map of 
\begin{align*}
    F_{(m_1, n_1), (m_2, n_2) } &: \Hom_{\Sk(M) \otimes_{\Sk(A)} \Sk(N)} \left((m_1, n_1), (m_2, n_2) \right) \\
    &\;\to \Hom_{\Sk(M \sqcup_A N)} \left( F(m_1, n_1), F(m_2, n_2) \right).
\end{align*}
If the map \(u \mapsto w\) is well defined it is the inverse of \(F_{(m_1, n_1), (m_2, n_2) }\), because \(F(f, g) \mapsto (f,g)\) and \(F(\iota_{m, a, n}) \mapsto \iota_{m,a,n}\).

Any equivalence of ribbon diagrams in \(\Sk(M \sqcup_A N)\) can be decomposed into equivalences which are fixed outside of one open set in the open cover of \((M \sqcup_A N) \times [0, 1]\). In particular this means that any isotopy of
\[u = u_{[1, t_N + \epsilon_N]} \circ u_{[t_N - \epsilon_N, t_N + \epsilon_N]} \circ u_{[t_{N-1} + \epsilon_{N-1}, t_N - \epsilon_N]} \circ \dots \circ u_{[t_1 - \epsilon_1, t_1 + \epsilon_1] \circ u_{[0, t_1 - \epsilon_1]}}\]
consists of the composition of equivalence of the following forms:
\begin{enumerate}
    \item \textbf{Equivalence of a non-crossing morphism.} Let \(u_i := u_{[t + \epsilon,t - \epsilon]}\) be a non-crossing ribbon diagram, so \(u_i = f \sqcup g\) for \(f \in F(M) \times [0, 1]\) and \(g \in F(N) \times [0, 1]\). The equivalences \(f \sim f'\) and \(g \sim g'\) of the ribbon tangles in \( F(M) \times [0, 1]\) and \(F(N) \times [0, 1]\) respectively define an equivalence of \(u_i\) to another non-crossing ribbon diagram \(u'_i := f' \sqcup g'\).
    \item \textbf{Equivalence in the middle region}. 
    \begin{figure}[H]
\floatbox[{\capbeside\thisfloatsetup{capbesideposition={left,top},capbesidewidth=7cm}}]{figure}[\FBwidth]
{\caption*{Let \(u_i := u_{[t- \epsilon, t + \epsilon]}\) be a crossing ribbon diagram, so \(u_i = v_1 \sqcup \Id_{F(m \lhd a, d \rhd n)}\). 
    Suppose we have an equivalence of \(v_1\) in the middle region \(v_1 \sim (r \sqcup s) \circ v_2 \circ (p \sqcup q)\) where \(r, p \in F(\emptyset \lhd A) \times [0, 1]\) and \(s, q \in F(A \rhd \emptyset) \times [0, 1]\) are depicted in the figure opposite. This defines an equivalence of \(u_i\) to 
    \(
    \left(r \sqcup s \sqcup \Id_{F(m \lhd a, d \rhd n)} \right) \circ \left( v_2 \sqcup \Id_{F(m \lhd a, d \rhd n)} \right)\)
    \(\circ \left(p \sqcup q \sqcup \Id_{F(m \lhd a, d \rhd n)} \right).
    \)
}}
{\myincludesvg{width=4cm}{}{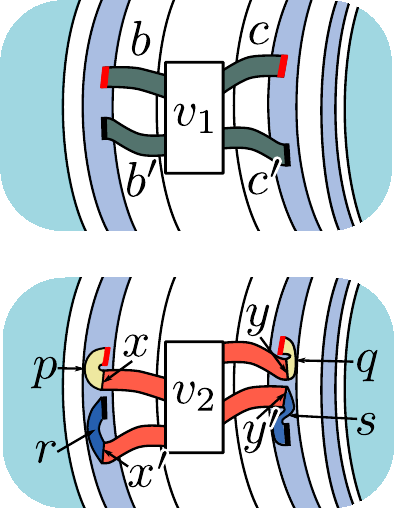}}
\end{figure}
    \item \textbf{Commuting with a crossing.} Let \(u_i := u_{[s,t - \epsilon]}\)\footnote{We use \(s\) as this \(u_i\) may only be part of one of the ribbon diagrams in the decomposition of \(u\).} be a non-crossing ribbon diagram of the form \(u_i = g \sqcup h \sqcup \Id_{F(b, c)}\) where \(g \in F(M) \times [0, 1]\) and \(h \in F(N) \times [0, 1]\) and let \(u_{i+1} := u_{[t- \epsilon, t+ \epsilon]}\) be a crossing ribbon such that \(v: b \sqcup c \to b' \sqcup c'\). There is an equivalence:
    \[
    \left( v \sqcup \Id_{F(m \lhd x, y \rhd n)} \right) \circ \left( g \sqcup h \sqcup \Id_{F(b, c)} \right) \sim  \left( g \sqcup h \sqcup \Id_{F(b', c')} \right) \circ \left( v \sqcup \Id_{F(m \lhd a, b \rhd n)} \right)
    \]
    which commutes these ribbon diagrams up to some modification of the identity components. 
    \item \paragraph{Merging crossings.}
    Let \(u_i\) and \(u_{i+1}\) both be crossing ribbon diagrams\footnote{To simplify the proof slightly, we assume that there are no points in the left crossing region which are not moved by the crossing.}, so 
    \(u_i = f \sqcup \Id_{F(m \lhd a, d \rhd n)}\) and \(u_{i+1} = g \sqcup \Id_{n \lhd b'', c'' \rhd m}\) for \(f: b \sqcup c \to b' \sqcup c'\) and \(g: x \sqcup y \to x' \sqcup y'\) for \(x = a \sqcup (b' - b'')\)\footnote{\((b' - b'')\) denotes set difference} and \(y = d \sqcup (c' - c'')\) (see the figure below). Then the composition \(u_{i+1} \circ u_i\) is equivalent to the single crossing \(u' := v \sqcup \Id_{F(m, n)}\) where \(v = \left( g \sqcup \Id_{b'' \sqcup c''} \right) \circ \left(f \sqcup \Id_{a \sqcup d} \right)\).
\begin{figure}[h]
    \centering
    \myincludesvg{height=3cm}{}{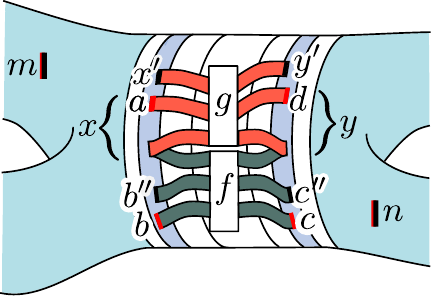}
    \label{fig:combine}
\end{figure}
\end{enumerate}

We shall now check that the map \(u \mapsto w\) is well defined by showing it is invariant under the equivalences listed above. 

\paragraph{Equivalence of a non-crossing morphism}
This is straightforward:
\(u_i := f \sqcup g \mapsto \left( E^{-1}_M(f), E^{-1}_N(g) \right)\) and \(u'_i := f' \sqcup g' \mapsto \left( E^{-1}_M(f'), E^{-1}_N(g') \right)\), but \(f \sim f'\) and \(g \sim g'\) implies \(E^{-1}_M(f) \sim E^{-1}_M(f')\) and \(E^{-1}_N(g) \sim E^{-1}_N(g')\), so these ribbon tangles map to the same morphism.

\paragraph{Equivalence of middle region}
\[
\begin{tikzcd}[column sep=-0em, row sep=0.5em]
(m \lhd (a \sqcup b), (c \sqcup d) \rhd n) 
        \ar[dd, "{(\rho_{m, a, b}, \Id)}"']
        \ar[rr, "{(\Id_m \lhd (\Id_a \sqcup p), (q \sqcup \Id_d) \rhd \Id_{n})}"]
    &
    & (m \lhd (a \sqcup x), (y \sqcup d) \rhd n)
        \ar[dd, "{(\rho_{m, a, x}, \Id)}"]
\\
    & \text{naturality of } \rho
    &
\\
((m \lhd a) \lhd b, (c \sqcup d) \rhd n)
        \ar[dd, "\iota_{m \lhd a, b, (c \sqcup d) \rhd n}"']
        \ar[rr, "{(\Id_{m \lhd a} \lhd p),  (q \sqcup \Id_d) \rhd \Id_{n})}"]
    &
    & ((m \lhd a) \lhd x, (y \sqcup d) \rhd n)
        \ar[dd, "\iota_{m \lhd a, x, (y \sqcup d) \rhd n}"]
\\
    & \text{naturality of } \iota
    &
\\
(m \lhd a, b \rhd ((c \sqcup d) \rhd n))
        \ar[dd, "{(\Id, \hat{u}_1)}"']
        \ar[rr, "{(\Id, p \rhd ((q \sqcup \Id_d) \rhd \Id_n))}"]
    &
    & (m \lhd a, x \rhd ((y \sqcup d) \rhd n))
        \ar[dd, "{(\Id, \hat{u}_2)}"]
\\
    & \text{definition of maps}
    &
\\
(m \lhd a, b' \rhd ((c' \sqcup d) \rhd n))
        \ar[dd, "\iota^{-1}_{m \lhd a, b', (c' \sqcup d) \rhd n}"']
        \ar[rr, "{(\Id_{m \lhd a, r \rhd ((s \sqcup \Id_d) \rhd \Id_n)})}"]
    &
    & (m \lhd a, x' \rhd ((y' \sqcup d) \rhd n))
        \ar[dd, "\iota^{-1}_{m \lhd a, x', (y' \sqcup d) \rhd n}"]
\\
    & \text{naturality of } \iota 
    &
\\
((m \lhd a) \lhd b', (c' \sqcup d) \rhd n)
        \ar[dd, "{(\rho^{-1}_{m, a, b'}, \Id)}"']
        \ar[rr, "{(\Id_{m \lhd a} \lhd r, (s \sqcup \Id_d) \rhd \Id_n)}"]
    &
    & ((m \lhd a) \lhd x', (y' \sqcup d) \rhd n)
        \ar[dd, "{(\rho^{-1}_{m, a, x'}, \Id)}"]
\\
    & \text{naturality of } \rho
    &
\\
(m \lhd (a \sqcup b'), (c' \sqcup d) \rhd n)
        \ar[rr, "{(\Id_m \lhd (\Id_a \sqcup r), (s \sqcup \Id_d) \rhd \Id_n)}"]
    &
    & (m \lhd (a \sqcup x'), (y' \sqcup d) \rhd n)
\end{tikzcd}
\]

\paragraph{Crossings commute with disjoint morphisms}
\[
\begin{tikzcd}[column sep=-3em, row sep=0.5em]
    (m \lhd (a \sqcup b), (c \sqcup d) \rhd n)) 
            \ar[rr, "{(g \sqcup \Id_{\emptyset \lhd b}, (\Id_{c \rhd \emptyset} \sqcup h))}"]
            \ar[dd, "{(\rho_{m, a, b}, \Id)}"']
        &
        & (m \lhd (x \sqcup b), (c \sqcup y) \rhd n)) 
            \ar[dd, "{(\rho_{m, x, b}, \Id)}"]
    \\
    & \text{naturality of } \rho
    &
    \\
    ((m \lhd a) \lhd b, (c \sqcup d) \rhd n)) 
            \ar[rr, "{(g \lhd \Id_b, \Id_{c \rhd \emptyset} \sqcup h)}"]
            \ar[dd, "\iota_{m \lhd a, b, (c \sqcup d) \rhd n}"']
        &
        & ((m \lhd x) \lhd b, (c \sqcup y) \rhd n)) 
            \ar[dd, "\iota_{m \lhd x, b, (c \sqcup y) \rhd n}"]
    \\
    & \text{naturality of } \iota
    &
    \\
    (m \lhd a, b \rhd ((c \sqcup d) \rhd n))
            \ar[rr, "{(g, \Id_b \rhd (\Id_{c \rhd \emptyset} \sqcup h))}"]
            \ar[dd, "{(\Id, \overline{v} \sqcup \Id_{\emptyset \rhd (d \rhd n)})}"']
        &
        & (m \lhd x, b \rhd ((c \sqcup y) \rhd n))
            \ar[dd, "{(\Id, \overline{v} \sqcup \Id_{\emptyset \rhd (y \rhd n)})}"]
    \\
    & \text{identity morphisms are central}
    &
    \\
    (m \lhd a, b' \rhd ((c' \sqcup d) \rhd n)) 
            \ar[rr, "{(g, \Id_{b'} \rhd (\Id_{c' \rhd \emptyset} \sqcup h))}"]
            \ar[dd, "\iota^{-1}_{m \lhd a, b', (c' \sqcup d) \rhd n}"']
        &
        & (m \lhd x, b' \rhd ((c' \sqcup y) \rhd n))
            \ar[dd, "\iota^{-1}_{m \lhd x, b', (c' \sqcup y) \rhd n}"]
    \\
    & \text{naturality of } \iota
    &
    \\
    ((m \lhd a) \lhd b', (c' \sqcup d) \rhd n)
            \ar[rr, "{(g \lhd \Id_{b'}, \Id_{c' \rhd \emptyset} \sqcup h)}"]
            \ar[dd, "{(\rho^{-1}_{m, a, b'}, \Id)}"']
        &
        & ((m \lhd x) \lhd b', (c' \sqcup y) \rhd n)
            \ar[dd, "{(\rho^{-1}_{m, x, b'}, \Id)}"]
    \\
    & \text{naturality of } \rho
    &
    \\
    (m \lhd (a \sqcup b'), (c' \sqcup d) \rhd n)
             \ar[rr, "{(g \sqcup \Id_{\emptyset \lhd b'}, (\Id_{c' \rhd \emptyset} \sqcup h))}"]
        &
        &(m \lhd (x \sqcup b'), (c' \sqcup y) \rhd n)
\end{tikzcd}
\]

\textbf{Merging Crossings}
\begin{figure}
\noindent\makebox[\textwidth]{
\begin{tikzcd}[row sep=3em, column sep=-1em, ampersand replacement=\&]
    (m \lhd (a \sqcup b), (c \sqcup d \sqcup e) \rhd n)
        \ar[rrr, "\iota_{m, a \sqcup b, (c \sqcup d \sqcup e \sqcup f) \rhd n}"]
        \ar[d, "{(\rho_{m, a, b}, \Id)}"]
    \&
    \&
    \& (m, (a \sqcup b) \rhd ( (c \sqcup d \sqcup e) \rhd n)) 
        \ar[d, "{(\Id, \rho_{a, b, (c \sqcup d \sqcup e) \rhd n})}"] 
        \\
    ((m \lhd a ) \lhd b, (c \sqcup d \sqcup e) \rhd n)
        \ar[rd, "\iota_{m \lhd a, b, (c \sqcup d \sqcup e) \rhd n}"']
    \&   \text{pentagon}
    \& 
    \& (m,  a \rhd ( b \rhd ( (c \sqcup d \sqcup e) \rhd n)))
        \ar[lld, "\iota^{-1}_{m, a, b \rhd ((c \sqcup d \sqcup e) \rhd n)}"]
        \ar[ddd, shift left=1.5em, "{(\Id_m, \Id_a \rhd \hat{f})}"{name=A}]
    \\
    \&  (m \lhd a , b \rhd ( (c \sqcup d \sqcup e) \rhd n))
        \ar[d, "{(\Id_m \lhd \Id_a, \hat{f})}"{name=B}]
        \ar[from=A, to=B, phantom, "\iota \text{ naturality}"]
    \&
    \&
    \\
    \& (m \lhd a , b' \rhd ( (c' \sqcup d \sqcup e) \rhd n))
        \ar[ld, "\iota^{-1}_{m \lhd a, b', (c' \sqcup d \sqcup e) \rhd n}"']
        \ar[rrd, "\iota_{m,  a, b' \rhd ( (c' \sqcup d \sqcup e) \rhd n)}"]
    \& 
    \&
    \\
    ((m \lhd a) \lhd b',  (c' \sqcup d \sqcup e) \rhd n)
        \ar[d, "{(\rho^{-1}_{m, a, b'} , \Id)}"]
    \&\text{pentagon}
    \& 
    \&(m , a \rhd (b' \rhd ( (c' \sqcup d \sqcup e) \rhd n)))
        \ar[d, "{(\Id, \rho^{-1}_{a, b', (c' \sqcup d \sqcup e) \rhd n})}"]
    \\
    (m \lhd (a \sqcup b')), (c' \sqcup d \sqcup e) \rhd n)
                \ar[rrr, "\iota_{m, a \sqcup b', (c' \sqcup d \sqcup e) \rhd n}"]
                \ar[dd, equal]
    \&
    \&
    \& (m, (a \sqcup b') \rhd (c' \sqcup d \sqcup e) \rhd n )
                \ar[dd, equal]
    \\[-1.5em]
    \& =
    \&
    \\[-1.5em]
    (m \lhd (b'' \sqcup x)), (y \sqcup c'' \sqcup e) \rhd n)
        \ar[rrr, "\iota_{m, b'' \sqcup x, (y \sqcup c'' \sqcup e) \rhd n}"]
        \ar[d, "{(\rho_{m, b'', x} , \Id)}"]
    \&
    \&
    \& (m, (b'' \sqcup x) \rhd (y \sqcup c'' \sqcup e) \rhd n )
        \ar[d, "{(\Id, \rho_{b'', x,  (y \sqcup c'' \sqcup e) \rhd n})}"]
    \\
    ((m \lhd b'') \lhd x, (y \sqcup c'' \sqcup e) \rhd n)
        \ar[rd, "\iota_{m \lhd b'', x, (y \sqcup c'' \sqcup e) \rhd n}"']
    \& \text{pentagon}
    \&
    \& (m, b'' \rhd (x \rhd (y \sqcup c'' \sqcup e) \rhd n ))
        \ar[lld, "\iota^{-1}_{m, b'', (y \sqcup c'' \sqcup e) \rhd n}"]
        \ar[ddd, shift left=1.5em, "{(\Id_{m}, \Id_{b''} \rhd \hat{g})}"{name=C}]
    \\
    \&(m \lhd b'', x \rhd ((y \sqcup c'' \sqcup e) \rhd n))
        \ar[d, "{(\Id_m \lhd \Id_{b''}, \hat{g})}"{name=D}]
        \ar[from=C, to=D, phantom, "\iota \text{ naturality}"]
    \&
    \&
    \\
    \&(m \lhd b'', x' \rhd ((y' \sqcup c'' \sqcup e) \rhd n))
        \ar[ld, "\iota^{-1}_{m \lhd b'', x', (y' \sqcup c'' \sqcup e) \rhd n}"']
        \ar[rrd, "\iota_{m, b'', x' \rhd ((y' \sqcup c'' \sqcup e) \rhd n)}"]
    \&
    \&
    \\
    ((m \lhd b'') \lhd x', (y' \sqcup c'' \sqcup e) \rhd n)
        \ar[d, "{(\rho^{-1}_{m, b'', x'}, \Id)}"]
    \& \text{pentagon}
    \&
    \& (m, b'' \rhd (x' \rhd ((y' \sqcup c'' \sqcup e) \rhd n)))
        \ar[d, "{(\Id, \rho^{-1}_{b'', x', (y' \sqcup c'' \sqcup e) \rhd n})}"]
    \\
    (m \lhd (b'' \sqcup x'), (y' \sqcup c'' \sqcup e) \rhd n)
    \&
    \&
    \& (m, (b'' \sqcup x') \rhd ((y' \sqcup c'' \sqcup e) \rhd n))
        \ar[lll, "\iota^{-1}_{m, b'' \sqcup x', (y' \sqcup c'' \sqcup e) \rhd n}"]
\end{tikzcd}}
\end{figure}
The composition of the morphism on the right of the diagram on the next page is \(u_{i+1} \circ u_i\). As this diagram is a commutative diagram: 
\begin{align*}
    u_{i+1} \circ u_i &=  \iota^{-1}_{m, b'' \sqcup x', (y' \sqcup c'' \sqcup e) \rhd n} \\
        &\quad \circ (\Id, \rho^{-1}_{b'', x', (y' \sqcup c'' \sqcup e) \rhd n}) \circ (\Id_{m}, \Id_{b''} \rhd \hat{g}) \circ (\Id, \rho_{b'', x,  (y \sqcup c'' \sqcup e) \rhd n}) \\
        &\quad \circ (\Id, \rho^{-1}_{a, b', (c' \sqcup d \sqcup e) \rhd n}) \circ (\Id_m \lhd \Id_a, \hat{f}) \circ (\Id, \rho_{a, b, (c \sqcup d \sqcup e) \rhd n}) \\
        &\quad \circ \iota_{m, a \sqcup b, (c \sqcup d \sqcup e \sqcup f) \rhd n} \\
        &= \iota^{-1}_{m, b'' \sqcup x', (y' \sqcup c'' \sqcup e) \rhd n} \\
        &\quad \circ (\Id, \Id_{b'' \rhd \emptyset} \sqcup \hat{g}) \circ (\Id, \rho^{-1}_{b'', x, (y \sqcup c'' \sqcup e) \rhd n}) \circ (\Id, \rho_{b'', x,  (y \sqcup c'' \sqcup e) \rhd n}) \\
        &\quad \circ (\Id, \Id_{a \rhd \emptyset} \sqcup \hat{f}) \circ (\Id, \rho^{-1}_{a, b, (c \sqcup d \sqcup e) \rhd n}) \circ (\Id, \rho_{a, b, (c \sqcup d \sqcup e) \rhd n}) \\
        &\quad \circ \iota_{m, a \sqcup b, (c \sqcup d \sqcup e \sqcup f) \rhd n} \\
        &\quad \text{ by naturality of } \rho \\
        &= \iota^{-1}_{m, b'' \sqcup x', (y' \sqcup c'' \sqcup e) \rhd n} \circ (\Id, \Id_{b'' \rhd \emptyset} \sqcup \hat{g}) \circ (\Id, \Id_{a \rhd \emptyset} \sqcup \hat{f}) \circ \iota_{m, a \sqcup b, (c \sqcup d \sqcup e \sqcup f) \rhd n} \\
        &= u'
\end{align*}
as required.
\end{proof}

\section{Skein Categories as \(k\)-linear Factorisation Homology}
\label{section:fact-homologies}
In this section, after briefly introducing factorisation homology, we shall prove that the \(k\)-linear (orientable) factorisation homology \(\int^{\Cat^{\times}_k}_{\_} \mathscr{V}\) is simply the skein category functor \(\Sk_{\mathscr{V}} (\_)\). Thus making precise the expected relation between factorisation homology and skein categories. 

\subsection{The Category \texorpdfstring{\(\Cat_k^{\times}\)}{Catktimes}}
\label{sec:catkdefn}
Firstly, we define \(\Cat_k^{\times}\).
\begin{defn}
A \emph{\((2,1)\)-category} \(\mathscr{C}\) is a \(2\)-category\footnote{Throughout this paper we shall assume all \(2\)-categories are strict i.e.\ categories enriched over \(\Cat\).} for which all \(2\)-morphisms have inverses.
\end{defn}

\begin{defn}
Let \(k\) be a commutative ring with identity. The category \(\kMod\) is the category of left \(k\)-modules and module homomorphisms. If \(k\) is a field then \(\kMod\) is \(\Vect\), the category of \(k\)-vector spaces and \(k\)-linear transformations.
\end{defn}

\begin{defn}
A \emph{\(k\)-linear category} is a category enriched over \(\kMod\) and a \emph{\(k\)-linear functor} is a \(\kMod\)-enriched functor.
\end{defn}

\begin{defn}
The \emph{category of \(k\)-linear categories} \(\Cat_k\) is the \((2,1)\)-category whose
\begin{enumerate}
    \item objects are small \(k\)-linear categories;
    \item \(1\)-morphisms are \(k\)-linear functors;
    \item \(2\)-morphisms are natural isomorphisms.
\end{enumerate}
\end{defn}

The \((2, 1)\)-category \(\Cat_k\) is a strict monoidal category with the categorical product \(\times\) as monoidal product:
\begin{enumerate}
    \item The product \(\mathscr{C} \times \mathscr{D}\) has as objects pairs \((m,n)\) where \(m \in \mathscr{C}\) and \(n \in \mathscr{D}\) and as morphisms pairs \((f,g)\) where \(f: m \to m'\) is a morphism in \(\mathscr{C}\) and \(g: n \to n'\) is a morphism in \(\mathscr{D}\).
    \item The monoidal unit \(1_{\Cat}\) is the category \(\catname{Pt}\) with a single object and a single morphism which is the identity morphism on this object
\end{enumerate}
\subsection{Factorisation Homology}
\label{section:FactorisationHomology}

We shall now define the \(n\)-dimensional factorisation homology \(\int_{M}^{\mathscr{C}^{\otimes}} \mathscr{E}\) which takes as inputs a framed \(E_n\)-algebra \(F_{\mathscr{E}}\) and an \(n\)-dimensional manifold \(M\). Except in this subsection, we shall only consider factorisation homologies of surfaces i.e. we shall assume \(n=2\). General introductory references for factorisation homology include Ginot \cite{Ginot} and Ayala and Francis~\cite{AyalaFrancis, AFPrimer}.

\subsubsection{Definition}
A factorisation homology is dependent on a choice of a monoidal \((\infty, 1)\)-category \(\mathcal{C}^{\otimes}\). We require that \(\mathscr{C}^{\otimes}\) is \(\otimes\)-presentable.

\begin{rmk}
An \((\infty, 1)\)-category is a category for which all \(k\)-morphisms for \(k>1\) are invertible. In the literature around factorisation homology, \((\infty, 1)\)-category is usually abbreviated to \(\infty\)-category.
\end{rmk}

\begin{defn}[{\cite[Definition~3.4]{AyalaFrancis}}]
A symmetric monoidal \((\infty,1)\)-category \(\mathscr{C}^{\otimes}\) is \emph{\(\otimes\)-presentable} if 
\begin{enumerate}
    \item \(\mathscr{C}\) is locally presentable with respect to an infinite cardinal \(\kappa\) i.e.\ is locally small, cocomplete and is generated under \(\kappa\)-filtered colimits by a set of \(\kappa\)-compact objects;
    \item the monoidal structure distributes over small colimits i.e.\ the functor \(C \otimes \_ : \mathscr{C} \to \mathscr{C}\) carries colimit diagrams to colimit diagrams.
\end{enumerate}
\end{defn}

\begin{rmk}
In this paper as we are dealing with surfaces, \(\mathscr{C}^{\otimes}\) will always be a \((2, 1)\)-category.
\end{rmk}

\begin{ex}
The \((2,1)\)-category \(\Cat_k^{\times}\) (see \cref{sec:catkdefn}) is \(\times\)-presentable \cites[Section~4]{KellyLack01}[7,115]{Kelly05}.
\end{ex}

We shall now define framed \(E_n\)-algebras which first requires us to define \(\Mfld{n}\) and its subcategory \(\Disc{n}\).

\begin{defn}
Let \(X\) and \(Y\) be smooth, oriented, finitary, \(n\)-dimensional manifolds. We denote by \(\Emb{n}(X,Y)\) the \(\infty\)-groupoid of the topological space of smooth oriented finitary embeddings of \(X\) into \(Y\) with the smooth compact open topology, i.e the objects of \(\Emb{n}(X, Y)\) are smooth oriented embeddings, the \(1\)-morphisms are isotopies, the \(2\)-morphisms are homotopies between the \(1\)-morphisms and so on. 
\end{defn}

\begin{defn}
Let \(\Mfld{n}\) be the symmetric monoidal \((\infty,1)\)-category whose
\begin{enumerate}
    \item objects are oriented \(n\)-dimensional manifolds;
    \item \(\Hom\)-space of morphisms between manifolds \(X\) and \(Y\) is the \(\infty\)-groupoid \(\Emb{n}(X, Y)\);
    \item symmetric monoidal product is disjoint union.
\end{enumerate}
\end{defn}

\begin{defn}
Let \(\Disc{n}\) be the full subcategory of \(\Mfld{n}\) of finite disjoint unions of \(\mathbb{R}^n\). Denote the inclusion functor by \(I: \Disc{n} \to \Mfld{n}\).
\end{defn}

\begin{defn}
A \emph{framed \(E_n\)-algebra} is a symmetric monoidal functor \(F_{\mathscr{E}}: \Disc{n} \to \mathscr{C}^{\otimes}\) where \(\mathscr{C}^{\otimes}\) is a symmetric monoidal \((\infty,1)\)-category. As \(F_{\mathscr{E}}\) is determined on objects by its value on a single disc, we define \(\mathscr{E} := F_{\mathscr{E}}(\mathbb{R}^n)\), and we use \(\mathscr{E}\) to refer to the associated framed \(E_n\)-algebra. 
\end{defn}

\begin{rmk}
Factorisation homologies can take more general \(n\)-disc algebras as coefficients. Let \(\catname{Mfld}^{G, \sqcup}_n\) be the category of smooth finitary \(n\)-dimensional manifolds with a \(G\)-structure: the Hom-space \(\catname{Emb}^{G}_n(X,Y)\) is the \(\infty\)-groupoid of \(G\)-framed embeddings of \(X\) into \(Y\). A \(\catname{Disc}^G_n\)-algebra is an symmetric monoidal functor \(F_{\mathscr{E}}: \catname{Disc}^{G,\sqcup}_n \to \mathscr{C}^{\otimes}\) where \(\catname{Disc}^{G,\sqcup}_n\) is the full subcategory of \(\catname{Mfld}^{G, \sqcup}_n\) of disjoint unions of \(G\)-framed discs. So a framed \(E_n\)-algebra is a \(\Disc{n}\)-algebra.
\end{rmk}

\begin{rmk}
The terminology framed \(E_n\)-algebra is somewhat confusing as there is also a notion of an \(E_n\)-algebra which is a \(\catname{Disc}^{\operatorname{fr}}_n\)-algebra. So a framed \(E_n\)-algebra relates to oriented discs and a \(E_n\)-algebra relates to framed discs. We shall only consider oriented manifolds and discs in this paper.
\end{rmk}

\begin{defn}
\label{defn:facthom}
Let \(\mathscr{C}^{\otimes}\) be a \(\otimes\)-presentable symmetric monoidal \((\infty,1)\)-category and let \(F: \Disc{n} \to \mathscr{C}^{\otimes}\) be a framed \(E_n\)-algebra with \(\mathscr{E} := F(\mathbb{R}^n)\). The left Kan extension of the diagram
\[
\begin{tikzcd}
\Disc{n} \arrow[r, "F"] \arrow[d, hook, "I"'] 
				&  \mathscr{C}^{\otimes} \\
\Mfld{n} \arrow[ru, dashrightarrow, "\int_{\_} \mathscr{E}"'] 
				& 
\end{tikzcd}
\]
is called the\footnote{As factorisation homology is defined via a universal construction we have uniqueness up to a contractible space of isomorphisms.} \emph{factorisation homology} with coefficients in \(\mathscr{E}\); its image on the manifold \(\Sigma\) is called the factorisation homology of \(\Sigma\) over \(\mathscr{E}\) and is denoted \(\int^{\mathscr{C}}_{\Sigma} \mathscr{E}\) or \(\int_{\Sigma} \mathscr{E}\).  
\end{defn}

\subsubsection{Excision}
Factorisation homology like classical homology satisfies an excision property: the factorisation homology of a cylinder gluing of two manifolds can be obtained from the factorisation homology of the original manifolds by tensoring relative to the submanifold glued along; hence, factorisation homology is determined locally. The relative tensor product is the relative tensor product of \(\otimes\)-presentable categories.
\begin{defn}
\label{defn:barrelativetensorproduct}
Let \(\mathscr{C}^{\otimes}\) be a \(\otimes\)-presentable symmetric monoidal \((\infty,1)\)-category. 
Let \(\mathscr{A}\) be an algebra object in \(\mathscr{C}^{\otimes}\), and let \(\mathscr{M}\) and \(\mathscr{N} \in \mathscr{C}^{\otimes}\) be right and left modules respectively over the algebra object \(\mathscr{A}\). The \emph{relative tensor product} \(\mathscr{M} \ftensor{\otimes}_{\mathscr{A}} \mathscr{N}\) is the colimit of the \(2\)-sided bar construction
\[
\begin{tikzcd}
\dots 
        \ar[r, shift left]
        \ar[r, shift right]
        \ar[r, shift left=1.68ex]
        \ar[r, shift right=1.68ex]
    &\mathscr{M} \otimes \mathscr{A} \otimes \mathscr{A} \otimes \mathscr{N}
        \ar[r, shift left=1.12ex]
        \ar[r]
        \ar[r, shift right=1.12ex]
    &\mathscr{M} \otimes \mathscr{A} \otimes \mathscr{N}
        \ar[r, shift left]
        \ar[r, shift right]
    & \mathscr{M} \otimes \mathscr{N}
\end{tikzcd}
\]
\end{defn}

\begin{figure}[h]
\label{fig:1}
    \centering
    \myincludesvg{width=0.33\textwidth}{}{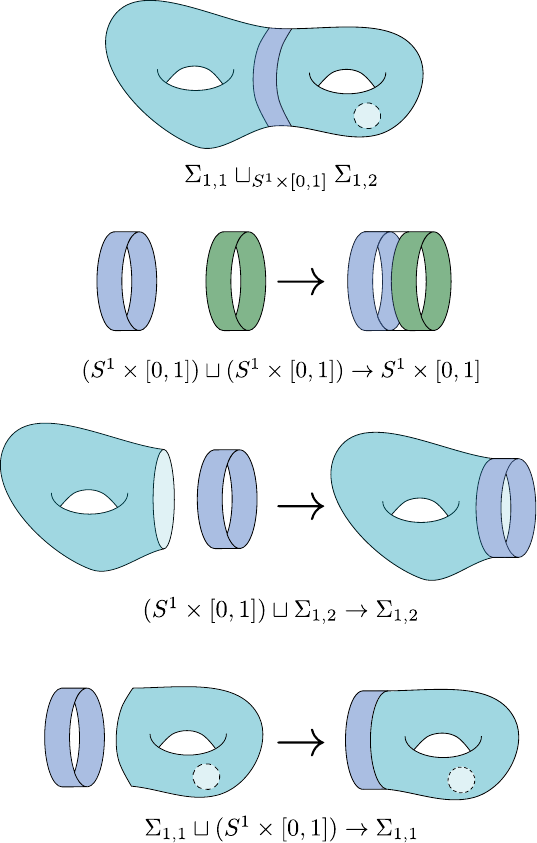}
    \caption{An example of the maps which induce the monoidal and module structures of the factorisation homologies.}
\end{figure}

Let \(\Sigma = M \sqcup_{C \times [0, 1]} N\) be the collar gluing of the \(n\)-dimensional manifolds \(M\) and \(N\) along \(C \times [0, 1]\) where \(C\) is an \((n-1)\)-dimensional manifold. The factorisation homology \(\int_{C \times [0, 1]} \mathscr{E}\) can be equipped with a monoidal structure induced by the embedding
\[\left(C \times [0,1] \right) \sqcup \left( C \times [0,1] \right) \xhookrightarrow{} C \times [0,1]\]
which retracts both copies of \(C \times [0,1]\) in the second coordinate and includes them into another copy of \(C \times [0,1]\).
The embeddings 
\begin{align*}
    M \sqcup (C \times [0,1]) &\xhookrightarrow{} M \\
    (C \times [0, 1]) \times N &\xhookrightarrow{} N
\end{align*}
induce a right \(\int_{C \times [0, 1]} \mathscr{E}\)-module structure on \(\int_M \mathscr{E}\) and a left \(\int_{C \times [0, 1]} \mathscr{E}\)-module structure on \(\int_N \mathscr{E}\). In other words, \(\int_{C \times [0, 1]} \mathscr{E}\) is an algebra object in \(\mathscr{C}^{\otimes}\), and \(\int_M \mathscr{E}\) and \(\int_N \mathscr{E}\) are right and left modules over this algebra object, so the relative tensor product \(\int_M \mathscr{E} \ftensor{\otimes}_{\left(\int_{C \times [0,1]} \mathscr{E} \right)} \int_{N} \mathscr{E}\) is defined.

\begin{thm}[{\cite[19]{AyalaFrancis}}]
\label{thm:excision}
Let \(\Sigma = M \sqcup_{C \times [0, 1]} N\) be the collar gluing of the \(n\)-dimensional manifolds \(M\) and \(N\) along \(C \times [0, 1]\) where \(C\) is a \((n-1)\)-dimensional manifold. There is an equivalence of categories 
\[
    \int_{M \sqcup_{C \times [0,1]} N} \mathscr{E} 
    \simeq \int_M \mathscr{E} \ftensor{\otimes}_{\int_{C \times [0, 1]} \mathscr{E}} \int_N \mathscr{E}.
\]
\end{thm}

\subsubsection{Characterisation of the Factorisation Homology of Surfaces}

\begin{thm}[{\cites{AyalaFrancis}[Theorem~2.5]{david1}}]
\label{thm:factcharacterised}
Let \(F_{\mathscr{E}}\) be a framed \(E_2\)-algebra in \(\mathscr{C}^{\otimes}\). The functor \(\int_{\_} \mathscr{E}\) is characterised by the following properties:
\begin{enumerate}
    \item If \(U\) is contractible then there is an equivalence in \(\mathscr{C}^{\otimes}\)
    \[\int_{U} \mathscr{E} \simeq \mathscr{E};\]
    \item If \(M \cong C \times [0, 1]\) for some \(1\)-manifold with corners \(C\) then the inclusion of intervals inside a larger interval induces a canonical framed \(E_1\)-structure on \(\int_M \mathscr{E}\).
    \item \(\int_{\_} \mathscr{E}\) satisfies excision (see \cref{thm:excision}).
\end{enumerate}
\end{thm}
\subsection{Equivalence of Relative Tensor Products}
\label{section:relative-tensor-products}
We would like to use \cref{thm:factcharacterised} to prove that the skein category \(\Sk_{\mathscr{V}}(\_)\) is equivalent to the \(k\)-linear factorisation homology \(\int^{\Cat_k^{\times}}_{\_} \mathscr{V}\). In \cref{sec:proofexcision} we proved that skein categories satisfy excision where the relative tensor product used is the Tambara relative tensor product of \(k\)-linear categories (\cref{defn:reltensorproduct}). However, the definition of excision used in the statement of excision of factorisation homology (\cref{thm:excision}) is the bar colimit relative tensor product of \(\otimes\)-presentable symmetric monoidal \((\infty, 1)\)-categories (\cref{defn:reltensorproduct}). We shall now prove that the Tambara relative tensor product \(\mathscr{M} \otimes_{\mathscr{A}} \mathscr{N}\) of \(k\)-linear categories is equivalent to the bar colimit relative tensor product \(\mathscr{M} \ftensor{\otimes}_{\mathscr{A}} \mathscr{N}\) where we consider \(\mathscr{M}, \mathscr{A}, \mathscr{N} \in \Cat_k^{\times}\). 

Firstly we note that as \(\Cat_k^{\times}\) is a \(2\)-category, we can truncate the bar construction after the second step:
\begin{defn}
Let \(\mathscr{A}\) be a monoidal \(k\)-linear category and let \(\mathscr{M}, \mathscr{N}\) be left/right \(\mathscr{A}\)-module \(k\)-linear categories. The \emph{truncated bar construction} is the diagram
\[\begin{tikzcd}
\mathscr{M} \otimes \mathscr{A} \otimes \mathscr{A} \otimes \mathscr{N}  \arrow[r, shift left=4, "G_1"] \arrow[r, "G_2"] \arrow[r, shift right=4, "G_3"]
    & \mathscr{M} \otimes \mathscr{A} \otimes \mathscr{N} \arrow[r, shift left=1.5, "F_1"] \arrow[r, shift right=1.5, "F_2"']
    & \mathscr{M} \otimes \mathscr{N}
\end{tikzcd}
\]
where \begin{alignat*}{3} 
&G_1: \mathscr{M} \otimes \mathscr{A} \otimes \mathscr{A} \otimes \mathscr{N}:\quad &&G_1(m,a,b,n) = (m \lhd a, b, n); \\
&G_2: \mathscr{M} \otimes \mathscr{A} \otimes \mathscr{A} \otimes \mathscr{N}: &&G_2(m,a,b,n) = (m, a \ast b, n); \\
&G_3: \mathscr{M} \otimes \mathscr{A} \otimes \mathscr{A} \otimes \mathscr{N}: &&G_3(m,a,b,n) = (m, a, b \rhd n); \\
&F_1: \mathscr{M} \otimes \mathscr{A} \otimes \mathscr{N}: &&F_1(m,a,n) = (m \lhd a, n); \\
&F_2: \mathscr{M} \otimes \mathscr{A} \otimes \mathscr{N}: &&F_2(m,a,n) = (m, a \rhd n);
\end{alignat*}
with \(m,a,b,n\) objects or morphisms in the categories \(\mathscr{M}, \mathscr{A}, \mathscr{A}, \mathscr{N}\) respectively, and there are two cells
\begin{align*}
    &\kappa_1: F_2 \circ G_1 \to F_1 \circ G_3\\
    &\kappa_2: F_1 \circ G_1 \to F_1 \circ G_2 \\
    &\kappa_3: F_2 \circ G_3 \to F_2 \circ G_2
\end{align*}
where \(\kappa_1\) is the identity and \(\kappa_2(n,a,b,m): ((n \lhd a) \lhd b, m) \to (n \lhd (a \ast b), m)\) and \(\kappa_3(n,a,b,m): (n, a \rhd ( b \rhd m)) \to (n, (a \ast b) \rhd m)\) are given by the associators of the \(\mathscr{A}\) action. 
\end{defn}

\subsubsection{Colimits of the Shape of the Truncated Bar Construction}
Before, considering the bicolimit of the truncated bar construction, we shall briefly consider bicolimits which have the shape of the truncated bar construction.

\begin{defn}
Let \(\mathscr{D}\) be the \(2\)-category 
\[ 
\begin{tikzcd}
\overline{A} 
        \ar[r, "\overline{g}_2"] 
        \arrow[r, shift left=4, "\overline{g}_1"] 
        \arrow[r, shift right=4, "\overline{g}_3"]
    & \overline{B} 
        \arrow[r, shift left, "\overline{f}_1"] 
        \arrow[r, shift right, "\overline{f}_2"'] 
    & \overline{C}
\end{tikzcd}
\]
with \(2\)-cells 
\[
\overline{\kappa}_1: \overline{f}_2 \circ \overline{g}_1 \to \overline{f}_1 \circ \overline{g}_3, \quad
\overline{\kappa}_2: \overline{f}_1 \circ \overline{g}_1 \to \overline{f}_1 \circ \overline{g}_2, \quad
\overline{\kappa}_3: \overline{f}_2 \circ \overline{g}_3 \to \overline{f}_2 \circ \overline{g}_2;
\]
and let \(X: \mathscr{D} \to \mathscr{C}\) be a (strict) \(2\)-functor to a \(2\)-category \(\mathscr{C}\). The images under \(X\) of objects, \(1\)-morphisms or \(2\)-morphisms are denoted without a bar, so
\[
X\left( \begin{tikzcd}
\overline{A} 
        \ar[r, "\overline{g}_2"] 
        \arrow[r, shift left=4, "\overline{g}_1"] 
        \arrow[r, shift right=4, "\overline{g}_3"]
    & \overline{B} 
        \arrow[r, shift left, "\overline{f}_1"] 
        \arrow[r, shift right, "\overline{f}_2"'] 
    & \overline{C}
\end{tikzcd} \right) 
=
\begin{tikzcd}
A \ar[r, "g_2"] \arrow[r, shift left=4, "g_1"] \arrow[r, shift right=4, "g_3"]& B \arrow[r, shift left, "f_1"] \arrow[r, shift right, "f_2"'] & C
\end{tikzcd} 
\text { and } X(\overline{\kappa}_i) = \kappa_i.
\]
\end{defn}

Recall the definition of a bicolimit:
\begin{defn}
Let \(\Diag\) denote the \((2,1)\)-category of diagrams of shape \(J\) in \(\mathscr{C}\):
\begin{enumerate}
    \item The objects are diagrams of shape \(J\) in \(\mathscr{C}\) i.e.\ lax functors \(X: J \to \mathscr{C}\);
    \item The \(1\)-morphisms are pseudonatural transformations;
    \item The \(2\)-morphisms are modifications.
\end{enumerate}
\end{defn}

\begin{defn}
Let \(\mathscr{X}\) denote the \((2,1)\)-category with a single object \(x\), a single \(1\)-morphism \(1_x\) and a single \(2\)-morphism which is the trivial \(2\)-cell \(1: 1_x \to 1_x\). The \emph{trace functor} on \(\mathscr{C}\) is the \(2\)-functor \(\Trace(x): \mathscr{C} \to \mathscr{X}\) which sends all objects in \(\mathscr{C}\) to \(x\), all \(1\)-morphisms to \(1_x\), and all \(2\)-morphisms to \(1\).
\end{defn}

\begin{defn}
Let \(\mathscr{C}\) be a \((2,1)\)-category. Denote by \(\Trace: \mathscr{C} \to \Diag\) the \(2\)-functor which sends an object \(x \in \mathscr{C}\) to the trace functor \(\Trace(x)\), a \(1\)-morphism \(f: x \to y\) to the trivial pseudonatural transformation \(\Gamma: \Trace(x) \to \Trace(y)\), and a \(2\)-morphism to the trivial modification \(\sigma: \Gamma \to \Gamma\).
\end{defn}

\begin{defn}
Let \(\mathscr{C}\) be a \(2\)-category and let \(X\) be a diagram of shape \(J\) in \(\mathscr{C}\). The \emph{\(2\)-colimit of \(X\)} is an object \(\Bicolim(X)\) in \(\mathscr{C}\) together with a pseudonatural equivalence 
\[\Gamma : \Hom_{\mathscr{C}}(\Bicolim(X), \_) \to \Diag(X, \Trace(\_)).\]
\end{defn}

This means that for all \(Y \in \mathscr{C}\) there is an equivalence of categories 
\[\Gamma_Y: \Hom_{\mathscr{C}}(\Bicolim(X), Y) \to \Diag(X, \Trace(Y)),\]
so in order to understand \(\Bicolim(X)\) we shall first look at \(\Diag(X, \Trace(Y))\).

\begin{prop}
The category \(\Diag(X, \Trace(Y))\) has objects of the form
\[
\sigma = \begin{pmatrix}
\sigma_A: A \to Y \\
\sigma_B: B \to Y \\
\sigma_C: C \to Y \\
\sigma_{f_i}: \sigma_B \to \sigma_C \circ f_i\\ 
\sigma_{g_j}: \sigma_A \to \sigma_B \circ g_j\\ 
\end{pmatrix}
\]
where \(i=1,2\) and \(j=1,2,3\), which satisfy the relations:
\begin{align*}
\sigma_C \kappa_1 = \Delta_{13} \Delta_{21}^{-1},\\ 
\sigma_C \kappa_2 = \Delta_{12} \Delta_{11}^{-1},\\ 
\sigma_C \kappa_3 = \Delta_{22} \Delta_{23}^{-1}
\end{align*}
where \(\Delta_{ij}:= (\sigma_{f_i} g_j) \sigma_{g_j} \). The morphisms of \(\Diag(X, \Trace(Y))\) are natural isomorphisms
\[\begin{tikzcd}
C 
    \ar[d, bend right=50, "\sigma_C"', ""{name=A}] 
    \ar[d, bend left=50, "\eta_C", ""'{name=B}]
        \ar[from=A, to=B, Rightarrow, "\Gamma"]
\\ 
Y,
\end{tikzcd}\]
satisfying the relations
\begin{align*}
\eta_{f_1}^{-1} (\Gamma f_1) \sigma_{f_1} &= \eta_{f_2}^{-1} (\Gamma f_2) \sigma_{f_2} \\
(\Delta^{\eta}_{ij})^{-1} (\Gamma f_i g_j) \Delta^{\sigma}_{ij} &= (\Delta^{\eta}_{kl})^{-1} (\Gamma f_k g_l) \Delta^{\sigma}_{kl}
\end{align*}
where \(i, k =1,2\) and \(j,l=1,2,3\).
\end{prop}

\begin{proof}
\sloppy This proof amounts to unravelling the definitions. An object of \(\Diag(X, \Trace(Y))\) is a pseudonatural transformation \(\sigma: X \to \Trace(Y)\). By the definition of a pseudonatural transformation we have
\begin{enumerate}
    \item for every \(\overline{X} \in \mathscr{D}\), that is \(\overline{X} = \overline{A}, \overline{B}, \overline{C}\), \(1\)-morphisms \(\sigma_{\overline{X}}: S(\overline{X}) \to \\Trace(Y)(\overline{X})\): we shall usually denote these morphisms simply as \(\sigma_{X}\) as they are morphisms \(\sigma_{X}: X \to Y\) for \(X = A, B, C\);
    \item for every pair of objects \((\overline{W}, \overline{X})\) in \(\mathscr{D}\), a natural transformation
    \[\sigma_{\overline{W}, \overline{X}}: (\sigma_{W})^* \circ \Trace(Y)(\overline{W}, \overline{X}) \to (\sigma_{\overline{X}})_* \circ S(\overline{W}, \overline{X}). \]
    This means that for every \(1\)-morphism \(\overline{h}: \overline{W} \to \overline{X}\) in \(\mathscr{D}\), that is \(\overline{h} = \overline{g}_1, \overline{g}_2, \overline{g}_3, \overline{f}_1, \overline{f}_2, \overline{f_i} \circ \overline{g_j}, 1_{W}\), there is a \(2\)-morphism
    \[\sigma_h : \sigma_W \to \sigma_X \circ h. \]
    As we are working with \((2,1)\)-categories, \(\sigma_h\) is automatically a \(2\)-isomorphism. As \(\sigma_{\overline{W}, \overline{X}}\) is natural, we have for every \(2\)-morphism \(\overline{\kappa}: \overline{h} \to \overline{l}\), that is \(\overline{\kappa} = \overline{\kappa}_1, \overline{\kappa}_2, \overline{\kappa}_3, \Id_{\overline{h}} \), that the following diagram commutes
    \[
    \begin{tikzcd}
    \sigma_W 
            \ar[d, "\sigma_h"]
            \ar[rd, "\sigma_l"]
        & \\
    \sigma_X \circ h 
            \ar[r, "\sigma_Y \circ \kappa"]
        & \sigma_X \circ l
    \end{tikzcd}
    \]
    This result is trivial for \(\Id_{\overline{h}}\), so let \(\overline{\kappa} = \overline{\kappa}_1, \overline{\kappa}_2\) or, \(\overline{\kappa}_3\). In which case, \(\overline{W}= \overline{A}\), \(\overline{X} = \overline{C}\), \(l = f_i \circ g_j\) and \(h = f_k \circ g_l\). As \(\sigma_h\) is invertible, we have that 
    \[\sigma_C \circ \kappa = \sigma_{f_i \circ g_j} \sigma^{-1}_{f_k \circ g_l}\]
    \item for every object \(\overline{X}\) of \(\mathscr{D}\), \(\sigma_{1_{\overline{X}}}\) is the identity natural isomorphism
    \item for every composition of morphisms \(f \circ g\) in \(\mathscr{D}\), \(\sigma_{f \circ g} = (\sigma_f g)(\sigma_g) \)
\end{enumerate}    
So we have that \(\sigma : X \to \Trace(y)\) consists of \(1\)-morphisms \(\sigma_X : X \to Y\) for \(X = A, B, C;\) and \(2\)-morphisms \(\sigma_{h}: \sigma_W \to \sigma_X \circ h\) for \(h= f_i, g_j : W \to X \) such that \(\pi_D \kappa_1 = \Delta_{13}  \Delta_{21}^{-1}\),
\(\pi_D \kappa_2 = \Delta_{12}  \Delta_{11}^{-1}\) and
\(\pi_D \kappa_3 = \Delta_{22} \Delta_{23}^{-1}\) where  \(\Delta_{ij} = (\sigma_{f_i} g_j) \sigma_{g_j}\).

A morphism \(\Gamma : \sigma \to \eta \) in \(\Diag(X, \Trace(Y))\) is a modification between \(\sigma\) and \(\eta\). The modification \(\Gamma\) assigns to each object \(\overline{X} \in \mathscr{D}\), that is \(\overline{X} = \overline{A}, \overline{B}, \overline{C}\), a \(2\)-morphism 
\[
\begin{tikzcd}[column sep=large]
X 
        \ar[r, bend left, "\sigma_X", ""'{name=source}]
        \ar[r, bend right, "\eta_X"', ""{name=target}]
        \ar[from=source, to=target, Rightarrow, "\Gamma_X", start anchor=north, end anchor=south]
    & Y 
\end{tikzcd}
\]
such that the following diagram commutes for all \(\overline{h}: \overline{W} \to \overline{X}\)
    \[
    \begin{tikzcd}
    \sigma_W 
            \ar[r, "\Gamma_W"]
            \ar[d, "\sigma_h"]
        & \eta_W 
            \ar[d, "\eta_h"] \\
    \sigma_X \circ h 
            \ar[r, "\Gamma_X h"]
        & \eta_X \circ h
    \end{tikzcd}
    \]
As all \(2\)-cells are invertible, applying this relation to the \(f_i\)s gives 
\[\Gamma_B = \eta_{f_i}^{-1} (\Gamma_C f_i) \sigma_{f_i}\]
and then applying this relation to the \(g_i\)s gives
\begin{align*}
    \Gamma_A &= \eta_{g_j}^{-1} (\Gamma_B g_j) \sigma_{g_j} \\
        &= \eta_{g_j}^{-1} \Big( \big(\eta_{f_i}^{-1} (\Gamma_C f_i) \sigma_{f_i}\big) g_j \Big) \sigma_{g_j} \\
        &\quad\text{ substituting } \Gamma_B = \eta_{f_i}^{-1} (\Gamma_C f_i) \sigma_{f_i} \\
        &= \eta_{g_j}^{-1} (\eta_{f_i}^{-1} g_j) (\Gamma_C f_i g_j) (\sigma_{f_i} g_j )  \sigma_{g_j}\\
        &\quad\text{ as }  \big(\eta_{f_i}^{-1} (\Gamma_C f_i) \sigma_{f_i}\big) g_j = (\eta_{f_i}^{-1} g_j) (\Gamma_C f_i g_j) (\sigma_{f_i} g_j ) \\
        &= (\Delta^{\eta}_{ij})^{-1} (\Gamma_C f_i g_j) \Delta^{\sigma}_{ij}
\end{align*}
from which we conclude that it is sufficient to define \(\Gamma_C\) and that the relation for \(g_j\) is automatically satisfied if it is for the compositions \(f_i \circ g_j\).
\end{proof}

\begin{rmk}
The morphisms \(\sigma_A, \sigma_B\) and \(\sigma_C\) fit into the diagram
\[\begin{tikzcd}
A 
        \ar[rd, bend right, "\sigma_A"'] 
        \ar[r, "g_2"] 
        \arrow[r, shift left=4, "g_1"] 
        \arrow[r, shift right=4, "g_3"]
    & B \ar[d, "\sigma_B"] 
        \arrow[r, shift left, "f_1"] 
        \arrow[r, shift right, "f_2"'] 
    & C 
        \ar[ld, bend left, "\sigma_C"] \\
    & Y 
    &,
\end{tikzcd} 
\]
and the natural isomorphisms \(\sigma_{f_i}\) and \(\sigma_{g_j}\) are \(2\)-cells in this diagram.
\end{rmk}
\subsubsection{Proof that the Relative Tensor Product is a Colimit of the Bar Construction}
We shall now prove that the Tambara relative tensor product of \(k\)-linear categories is the bicolimit in \(\Cat_k\) of the truncated bar construction.

\begin{thm}
\label{thm:relaresame}
The Tambara relative tensor product \(\mathscr{M} \otimes_{\mathscr{A}} \mathscr{N}\) of the right \(\mathscr{C}\)-module \(k\)-linear category \(\mathscr{M}\) and the left \(\mathscr{C}\)-module \(k\)-linear category \(\mathscr{N}\) relative to the \(k\)-linear monoidal category \(\mathscr{A}\) is the bicolimit of the truncated bar construction
\[\begin{tikzcd}
\mathscr{M} \otimes \mathscr{A} \otimes \mathscr{A} \otimes \mathscr{N}  \arrow[r, shift left=4, "G_1"] \arrow[r, "G_2"] \arrow[r, shift right=4, "G_3"]
    & \mathscr{M} \otimes \mathscr{A} \otimes \mathscr{N} \arrow[r, shift left=1.5, "F_1"] \arrow[r, shift right=1.5, "F_2"']
    & \mathscr{M} \otimes \mathscr{N}
\end{tikzcd}
\]
with \(2\)-cells \(\kappa_1\), \(\kappa_2\), \(\kappa_3\) defined above. Hence, there is a categorical equivalence 
\[
\mathscr{M} \ftensor{\otimes}_{\mathscr{A}} \mathscr{N} \simeq \mathscr{M} \otimes_{\mathscr{A}} \mathscr{N}.
\]
where \(\ftensor{\otimes}_{\mathscr{A}}\) is the bar colimit relative tensor product in \(\Cat_k^{\times}\) and \(\otimes_{\mathscr{A}}\) is the Tambara relative tensor product of \(k\)-linear categories.
\end{thm}
\begin{proof}
By the definition of a bicolimit there is an equivalence of categories
\[\Gamma_{\mathscr{C}}: \Cat_c(\Bicolim(X), \mathscr{C}) \to \DiagCat(X, \Trace(\mathscr{C})),\]
so if there is an equivalence of categories 
\[I_{\mathscr{C}}: \DiagCat(X, \Trace(\mathscr{C})) \to \bal(\mathscr{M}, \mathscr{N}; \mathscr{C})\]
for every \(\mathscr{C} \in \Cat_k\) then by \cref{defn:reltensorproduct} \(\Bicolim(X)\) is the relative tensor product \(\mathscr{M} \otimes_{\mathscr{A}} \mathscr{N}\). 

We shall now define \(I_{\mathscr{C}}\) and show it to be a equivalence of categories. Let \(\sigma\) be an object of \(\DiagCat(X, \Trace(\mathscr{C}))\), so
\[
\sigma = \begin{pmatrix}
\sigma_A: A \to Y \\
\sigma_B: B \to Y \\
\sigma_C: C \to Y \\
\sigma_{F_i}: \sigma_B \to \sigma_C \circ F_i\\ 
\sigma_{G_j}: \sigma_A \to \sigma_B \circ G_j\\ 
\end{pmatrix}
\]
where \(A:= \mathscr{M} \times \mathscr{A} \times \mathscr{A} \times \mathscr{N}\), \(B:=\mathscr{M} \times \mathscr{A} \times \mathscr{N}\) and \(C:=\mathscr{M} \times \mathscr{N}\). We define 
\[I_{\mathscr{C}}(\sigma) := \sigma_C  \text{ with balancing } \alpha: \sigma_C F_i \xRightarrow{} \sigma_C F_2: \quad \alpha := \sigma_{F_2} \sigma_{F_1}^{-1}.\]
A morphism of \(\DiagCat(X, \Trace(\mathscr{C}))\) is a natural isomorphism \(\Gamma: \sigma \to \eta\), we define \[I_{\mathscr{C}}(\Gamma):= \Gamma. \]

\textbf{\(I_{\mathscr{C}}\) is a well-defined functor}

This requires one to prove two things: \(\alpha\) is an \(\mathscr{A}\)-balancing and \(\Gamma\) is a natural transformation of \(\mathscr{A}\)-balanced functors. To show \(\alpha\) is a balancing of \(\sigma_C\) we must show that the diagram
\[
\begin{tikzpicture}[commutative diagrams/every diagram]
\node (P0) at (90:2.3cm) {\(((m \lhd a) \lhd b, n) \) };
\node (P1) at (90+72:2cm) {\((m \lhd a, b \rhd n)\)} ;
\node (P2) at (90+2*72:2cm) {\makebox[5ex][r]{\((m, a \rhd(b \rhd n))\)}};
\node (P3) at (90+3*72:2cm) {\makebox[5ex][l]{\((m, (a \ast b) \rhd n)\)}};
\node (P4) at (90+4*72:2cm) {\((m \lhd (a \ast b), n) \)};
\path[commutative diagrams/.cd, every arrow, every label]
(P0) edge node[swap] {\(\alpha_{m \lhd a, b, n} = (\alpha G_1)(m, a, b, n)\)} (P1)
(P1) edge node[swap] {\(\alpha_{m,a,b \rhd n} = (\alpha G_3)(m, a, b, n)\)} (P2)
(P2) edge node {\(\sigma_C (\kappa_3)_{m, a, b, n}\)} (P3)
(P4) edge node {\(\alpha_{m, a \ast b, n} = (\alpha G_2)(m, a, b, n)\)} (P3)
(P0) edge node {\(\sigma_C (\kappa_2)_{m, a, b, n}\)} (P4);
\end{tikzpicture}
\]
commutes for all \((m, a, b, n) \in \mathscr{M} \times \mathscr{A} \times \mathscr{A} \times \mathscr{N}\). This is the case as 
\begin{align*}
&(\sigma_C \kappa_3)(\alpha G_3)(\sigma_C \kappa_1)(\alpha G_1) (\sigma_C \kappa_2^{-1}) \\
&= \Delta_{22}\Delta_{23}^{-1} (\sigma_{F_2} G_3)(\sigma_{F_1}^{-1} G_3) \Delta_{13} \Delta_{21}^{-1} (\sigma_{F_2} G_1)(\sigma_{F_1}^{-1}G_1) \Delta_{11} \Delta_{12}^{-1} \\
&\quad \text{ by definition of } \alpha \text{ and compatibility relations of } \sigma \\
&= (\sigma_{F_2} G_2) \sigma_{G_2} \sigma_{G_3}^{-1} (\sigma_{F_2}^{-1} G_3) (\sigma_{F_2} G_3)(\sigma_{F_1}^{-1} G_3) (\sigma_{F_1} G_3) \sigma_{G_3} \sigma_{G_1}^{-1}\\
&\quad(\sigma_{F_2}^{-1} G_1) (\sigma_{F_2} G_1)(\sigma_{F_1}^{-1}G_1) (\sigma_{F_1} G_1) \sigma_{G_1} \sigma_{G_2}^{-1} (\sigma_{F_1}^{-1} G_2) \\
&\quad \text{ by definition of } \Delta_{ij} \\
&= (\sigma_{F_2} G_2) (\sigma_{F_1}^{-1} G_2) \\
&\quad \text{ cancelling terms} \\
&= (\sigma_{F_2}\sigma_{F_1}^{-1}) G_2 \\
&= \alpha G_2
\end{align*}
Hence, \(\sigma_C\) is \(\mathscr{A}\)-balanced with balancing \(\alpha\). 

Now we shall show that the natural transformation \(\Gamma_C: \sigma_{C} \to \eta_C\) is a natural transformation of \(\mathscr{A}\)-balanced functors. To show this me must show that that following diagram commutes:
\[
\begin{tikzcd}[column sep=3em, row sep=3em]
\sigma_C(m \lhd a, n) 
        \ar[r, "(\alpha_{\sigma})_{m, a, n}"]
        \ar[d, "(\Gamma_C)_{(m \lhd a, n)}"]
    & \sigma_C (m, a \rhd n) 
        \ar[d, "(\Gamma_C)_{(m, a \rhd n)}"] \\
\eta_C(m \lhd a, n) 
        \ar[r, "(\alpha_{\eta})_{m, a, n}"]
    & \eta_C (m, a \rhd n) 
\end{tikzcd}
\]
This is the case as by the definition of \(\Gamma\) we have that
\begin{align*}
    &\eta_{F_1}^{-1} (\Gamma_C F_1) \sigma_{F_1} = \eta_{F_2}^{-1}(\Gamma_C F_2) \sigma_{F_2} \\
\implies &\eta_{F_2} \eta_{F_1}^{-1} (\Gamma_C F_1) = (\Gamma_C F_2) \sigma_{F_2} \sigma^{-1}_{F_1} \\
\implies &\alpha_{\eta} (\Gamma_C F_1) = (\Gamma_C F_2) \alpha_{\sigma}. 
\end{align*}
Thus, \(\Gamma_C: \sigma_{C} \to \eta_C\) is a natural transformation of \(\mathscr{A}\)-balanced functors, and we have concluded the proof that \(I_{\mathscr{C}}\) is well-defined. 

\textbf{\(I_{\mathscr{C}}\) is surjective}

Let \(F : \mathscr{M} \times \mathscr{N} \to \mathscr{C}\) be an \(\mathscr{A}\)-balanced functor with balancing \(\alpha\) i.e.\ \(F\) is an object of \(\bal(\mathscr{M}, \mathscr{N}; \mathscr{C})\). Define
\begin{equation}
\label{eq:Ifull}
    \sigma = 
\begin{pmatrix*}[l]
\sigma_{A} : A \to \mathscr{C} \text{ is } F \\
\sigma_{B} : B \to \mathscr{C} \text{ is } F F_1 \\
\sigma_{C} : C \to \mathscr{C} \text{ is } F F_1 G_2 \\
\sigma_{F_1} : F F_1 \to F F_2 \text{ is the identity} \\
\sigma_{F_2}: F F_1 \to F F_2 \text{ is } \alpha \\
\sigma_{G_1}: F F_1 G_2 \to F F_1 G_1 \text{ is } F \kappa_2^{-1} \\
\sigma_{G_2}: F F_1 G_2 \to F F_1 G_2 \text{ is the identity} \\
\sigma_{G_3}: F F_1 G_2 \to F F_1 G_3 \text{ is } (\alpha^{-1} G_3)(F \kappa_3^{-1}) (\alpha G_2)
\end{pmatrix*}
\end{equation}

If \(\sigma\) is a well-defined element of \(\DiagCat(X, \Trace(\mathscr{C}))\) then \(I_{\mathscr{C}}(\sigma) = F\), so it remains to show that \(\kappa_1 = \Delta_{13} \Delta_{21}^{-1}\), \(\sigma_C \kappa_2 = \Delta_{12} \Delta_{11}^{-1}\) and \(\sigma_C \kappa_3 = \Delta_{22} \Delta_{23}^{-1}\) where \(\Delta_{ij}:= (\sigma_{F_i} G_j) \sigma_{G_j}\):
\begin{align*}
    \Delta_{13} \Delta_{21}^{-1} &= (\sigma_{F_1} G_2) \sigma_{G_2} \sigma_{G_1}^{-1} (\sigma_{F_1}^{-1} G_1) \text{ by definition of } \Delta_{ij} \\
        &= F\kappa_2 \text{ by definition of } \sigma_{F_i}, \sigma_{G_j}. \\
    \Delta_{22} \Delta_{23}^{-1} &= (\sigma_{F_2} G_2) \sigma_{G_2} \sigma_{G_3}^{-1} (\sigma_{F_2}^{-1} G_3) \text{ by definition of } \Delta_{ij} \\
        &= (\alpha G_2) (\alpha^{-1} G_2) (F \kappa_3) (\alpha G_3) (\alpha^{-1} G_3) \text{ by definition of } \sigma_{F_i}, \sigma_{G_j}. \\
        &= F \kappa_3. \\
    \Delta_{13}\Delta_{21}^{-1} &= (\sigma_{F_1} G_3) \sigma_{G_3} \sigma_{G_1}^{-1} (\sigma_{F_2}^{-1} G_1) \text{ by definition of } \Delta_{ij} \\
        &=  (\alpha^{-1} G_3)(F \kappa_3^{-1}) (\alpha G_2) (F \kappa_2) (\alpha^{-1} G_1) \text{ by 
        definition of } \sigma_{F_i}, \sigma_{G_j} \\
        &= F \kappa_1 \text{ i.e.\ identity, by pentagon of } \alpha.
\end{align*}

\textbf{\(I_{\mathscr{C}}\) is full and faithful.}

Suppose \(I_{\mathscr{C}}(\Gamma) = I_{\mathscr{C}}(\Xi)\). By definition of \(I_{\mathscr{C}}\) this is \(\Gamma = \Xi\); hence, \(I_{\mathscr{C}}\) is faithful.

\sloppy Let \(\xi: F \xRightarrow{} G\) be a \(\mathscr{A}\)-balanced natural transformation between the \(\mathscr{A}\)-balanced functors \(F, G: \mathscr{M} \times \mathscr{N} \to \mathscr{A}\), i.e.\ \(\xi\) is a morphism of \(\bal(\mathscr{M}, \mathscr{N}; \mathscr{C})\). We have already shown \(I_{\mathscr{C}}\) to be surjective, so we have \(\sigma\) and \(\eta\) such that \(I_{\mathscr{C}}(\sigma) = F\) and \(I_{\mathscr{C}}(\eta) = G\) where \(\sigma\) is defined in \cref{eq:Ifull} and \(\eta\) is defined analogously. In order to show that \(I_{\mathscr{C}}\) is full we must find a morphism \(\Gamma: \sigma \to \mu\) in 
 \(\DiagCat(X, \Trace(\mathscr{C}))\) such that \(I_{\mathscr{C}}(\Gamma) = \xi\). 
 
 Define
 \[\Gamma = \left( 
 \begin{tikzcd}[column sep=4em]
 \mathscr{M} \times \mathscr{N} 
    \ar[d, bend left, start anchor={[xshift=-0.7em, yshift=-0.5em]east}, end anchor={east}, "\eta_C", ""'{name=t}]
    \ar[d, bend right, start anchor={[xshift=0.7em, yshift=-0.5em]west}, end anchor={west}, "\sigma_C"', ""{name=s}]
    \ar[from=s, to=t, Rightarrow, "{\Gamma_C:= \xi}"]
 \\
\mathscr{C}
 \end{tikzcd}
 \right).\]
 As \(I_{\mathscr{C}}(\Gamma)= \xi\), it remains to show \(\Gamma\) is a well-defined morphism in \(\DiagCat(X, \Trace(\mathscr{C}))\) i.e.\ that 
 \begin{enumerate}
     \item \(\eta_{F_1}^{-1} (\Gamma_C F_1) \sigma_{F_1} = \eta_{F_2}^{-1} (\Gamma_C F_2) \sigma_{F_2} \) and
     \item \(
    (\Delta^{\eta}_{ij})^{-1} (\Gamma_C F_i G_j) \Delta^{\sigma}_{ij} = (\Delta^{\eta}_{kl})^{-1} (\Gamma_C F_k G_l) \Delta^{\sigma}_{kl}\) for all \(i, k =1,2; j,l=1,2,3\). 
 \end{enumerate}
 \begin{enumerate}
     \item As \(\xi\) is an \(\mathscr{A}\)-balanced natural transformation
     \begin{align*}
          &(\eta_{F_2} \eta_{F_1}^{-1})_{m,a,n} \xi_{(m \lhd a, n)} = \xi_{(m, a \rhd n)} (\sigma_{F_2} \sigma_{F_1}^{-1})_{m, a, n} \\
          &\quad\text{ where } \sigma_{F_2} \sigma_{F_1}^{-1} \text{ is the balancing of } I_{\mathscr{C}}(\sigma) \\
          &\quad\text{ and } \eta_{F_2} \eta_{F_1}^{-1} \text{ is the balancing of } I_{\mathscr{C}}(\eta) \\
        \implies &(\eta_{F_1}^{-1})_{m,a,n} \xi_{(m \lhd a, n)}  (\sigma_{F_1})_{m, a, n} = (\eta_{F_2}^{-1})_{m,a,n} \xi_{(m, a \rhd n)} (\sigma_{F_2})_{m, a, n} \\       
        \implies & \eta_{F_1}^{-1} (\Gamma_C F_1)  \sigma_{F_1} (m, a, n) = \eta_{F_2}^{-1} (\Gamma_C F_2) \sigma_{F_2} (m, a, n) \\
        \implies & \eta_{F_1}^{-1} (\Gamma_C F_1)  \sigma_{F_1} = \eta_{F_2}^{-1} (\Gamma_C F_2) \sigma_{F_2}
     \end{align*}
    \item Denote \(\operatorname{Eq}(i,j) := (\Delta^{\eta}_{ij})^{-1} (\Gamma_C F_i G_j) \Delta^{\sigma}_{ij}\). By definition of \(\Delta_{i,j}\), \(\sigma\) and \(\eta\) we have that
    \begin{align*}
        \Delta^{\sigma}_{1j} &= (\sigma_{F_1} G_j) \sigma_{G_j} = \sigma_{G_j} \\
        \Delta^{\eta}_{1j} &= (\eta_{F_i} G_j) \eta_{G_j} = \eta_{G_j} \\
        \Delta^{\sigma}_{2j} &= (\sigma_{F_1} G_j) \sigma_{G_j} =  (\alpha_{\sigma} G_j) \sigma_{G_j} \\
        \Delta^{\eta}_{2j} &= (\eta_{F_1} G_j) \eta_{G_j} =  (\alpha_{\eta} G_j) \eta_{G_j} \\
    \end{align*}
    So 
    \begin{align*}
        \operatorname{Eq}(1,j) &= \eta^{-1}_{G_j} (\xi F_1 G_j) \sigma_{G_j} \text{ and } \\
        \operatorname{Eq}(2,j) &= \eta^{-1}_{G_j} (\alpha^{-1}_{\eta} G_j) (\xi F_2 G_j) (\alpha_{\sigma} G_j) \sigma_{G_j} \\
        &= \eta^{-1}_{G_j} \left( \left( \alpha^{-1}_{\eta} (\xi F_2) \alpha_{\sigma}\right) G_j \right)\sigma_{G_j} \\
        &= \eta^{-1}_{G_j} \left( \xi F_1 G_j \right)\sigma_{G_j} \text{ as } \xi \text{ is } \mathscr{A}\text{-balanced}\\
        &= \operatorname{Eq}(1,j).
    \end{align*}
    This means it remains to show \(\operatorname{Eq}(1,1) = \operatorname{Eq}(1,2) = \operatorname{Eq}(1,3)\):
    \begin{align*}
        \operatorname{Eq}(1,2) &= \eta^{-1}_{G_2} (\xi F_1 G_2) \\
            &= (\xi F_1 G_2)\\
        \operatorname{Eq}(1,1) &= \eta^{-1}_{G_1} (\xi F_1 G_1) \\
            &= (\eta_C \kappa_2) (\xi F_1 G_1) (\sigma_C \kappa_2^{-1}) \\
            &= 
\begin{tikzcd}[column sep=6em, ampersand replacement=\&]
    \vphantom{ty}\mathscr{M} \times \mathscr{A} \times \mathscr{A} \times \mathscr{N}
        \ar[r, bend left=70, looseness=2, "F_1 G_2", ""{name=s1, below},
            start anchor={[yshift=0.7ex]east},
            end anchor={[yshift=0.7ex]west}]
        \ar[r, "F_1 G_1" description, ""{name=t1}, ""'{name=s2}]
        \ar[r, bend right=70, looseness=2, "F_1 G_2"{below}, ""{name=t2},
            start anchor={[yshift=-0.7ex]east},
            end anchor={[yshift=-0.7ex]west}]
        \ar[from=s1, to=t1, Rightarrow, "\,\kappa_2^{-1}",
            start anchor={[yshift=0.2ex]south},
            end anchor=north]
        \ar[from=s2, to=t2, Rightarrow, "\,\kappa_2",
           start anchor={south},
           end anchor=center]
    \& \vphantom{ty}\mathscr{M} \times \mathscr{N}
        \ar[r, bend left=50, "\sigma_C", ""{name=s3, below},
            start anchor={[yshift=0.7ex]east},
            end anchor={[yshift=0.7ex]west}]
        \ar[r, bend right=50, "\eta_C"{below}, ""{name=t3},
            start anchor={[yshift=-0.7ex]east},
            end anchor={[yshift=-0.7ex]west}]
        \ar[from=s3, to=t3, Rightarrow, "\,\xi",
            start anchor={[yshift=0.2ex]south},
            end anchor=center]
    \& \vphantom{ty}\mathscr{C}
\end{tikzcd} \\
            &=
\begin{tikzcd}[column sep=5em, ampersand replacement=\&]
    \vphantom{ty}\mathscr{M} \times \mathscr{A} \times \mathscr{A} \times \mathscr{N}
        \ar[r, "F_1 G_2", ""{name=t1}, ""'{name=s2}]
    \& \vphantom{ty}\mathscr{M} \times \mathscr{N}
        \ar[r, bend left=50, "\sigma_C", ""{name=s3, below},
            start anchor={[yshift=0.7ex]east},
            end anchor={[yshift=0.7ex]west}]
        \ar[r, bend right=50, "\eta_C"{below}, ""{name=t3},
            start anchor={[yshift=-0.7ex]east},
            end anchor={[yshift=-0.7ex]west}]
        \ar[from=s3, to=t3, Rightarrow, "\,\xi",
            start anchor={[yshift=0.2ex]south},
            end anchor=center]
    \& \vphantom{ty}\mathscr{C}
\end{tikzcd} \\
            &= (\xi F_1 G_2) \\
            &= \operatorname{Eq}(1,2) \\
    \operatorname{Eq}(1,3) &= (\alpha_{\eta}^{-1} G_2) (\eta_C \kappa_3)(\alpha_{\eta} G_3)(\xi F_1 G_3) (\alpha_{\sigma}^{-1} G_3)(\sigma_C \kappa_3^{-1})(\alpha_{\sigma} G_2) \\
            &= (\alpha_{\eta}^{-1} G_2) (\eta_C \kappa_3)\big( \left(\alpha_{\eta}(\xi F_1) (\alpha_{\sigma}^{-1} \right) G_3 \big)(\sigma_C \kappa_3^{-1})(\alpha_{\sigma} G_2) \\
            &= (\alpha_{\eta}^{-1} G_2) (\eta_C \kappa_3)(\xi F_2 G_3)(\sigma_C \kappa_3^{-1})(\alpha_{\sigma} G_2) \text{ as } \xi \text{ is } \mathscr{A}\text{-balanced}\\
            &= (\alpha_{\eta}^{-1} G_2) (\xi F_2 G_2) (\alpha_{\sigma} G_2) \text{ by } \left( \dagger \right) \\
            &= \xi F_1 G_2 \text{ as } \xi \text{ is } \mathscr{A}\text{-balanced} \\
            &= \operatorname{Eq}(1,2)
    \end{align*}
where \(\left( \dagger \right)\):
\begin{align*}
    &(\eta_C \kappa_3) (\xi F_2 G_3) (\sigma_C \kappa_3^{-1}) \\
    \quad&= 
\begin{tikzcd}[column sep=6em, ampersand replacement=\&]
    \vphantom{ty}\mathscr{M} \times \mathscr{A} \times \mathscr{A} \times \mathscr{N}
        \ar[r, bend left=70, looseness=2, "F_2 G_2", ""{name=s1, below},
            start anchor={[yshift=0.7ex]east},
            end anchor={[yshift=0.7ex]west}]
        \ar[r, "F_2 G_3" description, ""{name=t1}, ""'{name=s2}]
        \ar[r, bend right=70, looseness=2, "F_2 G_2"{below}, ""{name=t2},
            start anchor={[yshift=-0.7ex]east},
            end anchor={[yshift=-0.7ex]west}]
        \ar[from=s1, to=t1, Rightarrow, "\,\kappa_3^{-1}",
            start anchor={[yshift=0.2ex]south},
            end anchor=north]
        \ar[from=s2, to=t2, Rightarrow, "\,\kappa_3",
           start anchor={south},
           end anchor=center]
    \& \vphantom{ty}\mathscr{M} \times \mathscr{N}
        \ar[r, bend left=50, "\sigma_C", ""{name=s3, below},
            start anchor={[yshift=0.7ex]east},
            end anchor={[yshift=0.7ex]west}]
        \ar[r, bend right=50, "\eta_C"{below}, ""{name=t3},
            start anchor={[yshift=-0.7ex]east},
            end anchor={[yshift=-0.7ex]west}]
        \ar[from=s3, to=t3, Rightarrow, "\,\xi",
            start anchor={[yshift=0.2ex]south},
            end anchor=center]
    \& \vphantom{ty}\mathscr{C}
\end{tikzcd} \\
            \quad&=
\begin{tikzcd}[column sep=5em, ampersand replacement=\&]
    \vphantom{ty}\mathscr{M} \times \mathscr{A} \times \mathscr{A} \times \mathscr{N}
        \ar[r, "F_2 G_2", ""{name=t1}, ""'{name=s2}]
    \& \vphantom{ty}\mathscr{M} \times \mathscr{N}
        \ar[r, bend left=50, "\sigma_C", ""{name=s3, below},
            start anchor={[yshift=0.7ex]east},
            end anchor={[yshift=0.7ex]west}]
        \ar[r, bend right=50, "\eta_C"{below}, ""{name=t3},
            start anchor={[yshift=-0.7ex]east},
            end anchor={[yshift=-0.7ex]west}]
        \ar[from=s3, to=t3, Rightarrow, "\,\xi",
            start anchor={[yshift=0.2ex]south},
            end anchor=center]
    \& \vphantom{ty}\mathscr{C}
\end{tikzcd} \\ 
    \quad&= \xi F_2 G_2
\end{align*}
 \end{enumerate}
\end{proof}
\subsection{Equivalence of Skein Categories and \(k\)--linear Factorisation Homology of Surfaces}
We now shall combine the results proven so far in this paper with the characterisation of factorisation homology in \cref{thm:factcharacterised} to prove the following:
\begin{thm}
\label{thm:skeinklinear}
Let \(\mathscr{V}\) be \(k\)-linear strict ribbon category \(\mathscr{V}\). The functor
\[\Sk_{\mathscr{V}}(\_): \Mfld{2} \to \Cat_k^{\times}\]
is the \(k\)-linear factorisation homology 
\[\int^{\Cat_k^{\times}}_{\_} \mathscr{V}: \Mfld{2} \to \Cat_k^{\times}\]
of surfaces with coefficients in \(\mathscr{V}\).
\end{thm}
\begin{proof}
Firstly we note that as \(\mathscr{V}\) is a ribbon category it defines a framed \(E_2\)-algebra \(F_{\mathscr{V}}: \Mfld{2} \to \Cat_k^{\times}\) such that \(F_{\mathscr{V}}(\mathbb{D}) = \mathscr{V}\), so the \(k\)-linear factorisation with coefficients in \(\mathscr{V}\) is well defined. 

We saw in \cref{rmk:embedding} that an embedding of surfaces \(\Sigma \xhookrightarrow{} \Pi\) induces a functor \(\Sk(\Sigma) \to \Sk(\Pi)\) between their skein categories, and in \cref{rmk:naturalityofiso} that isotopies of embeddings define natural transformations. This implies that
    \[
    \Sk_{\mathscr{V}}(\_): \Mfld{2} \to \Cat_k
    \]
    is a \(2\)-functor. So it remains to apply \cref{thm:factcharacterised}:
\begin{enumerate}
    \item From \cref{cor:disc} we have an equivalence of categories \(\Sk_{\mathscr{V}}(\mathbb{D}^2) \simeq \mathscr{V}\).
    \item From \cref{rmk:canonicalmonoidal} we have for any \(1\)-manifold \(C\) that \(\Sk(C \times [0,1] )\) has a canonical monoidal structure induced from the inclusions of intervals.
    \item In \cref{thm:skeinexcision} we have proven that given suitable thickened embeddings there is an equivalence of categories
    \[ 
    \Sk_{\mathscr{V}}\left( M \sqcup_{A} N \right) \simeq \Sk_{\mathscr{V}}(M) \otimes_{\Sk_{\mathscr{V}}(A)} \Sk_{\mathscr{V}}(N).
    \]
    where \(\Sk_{\mathscr{V}}(M) \otimes_{\Sk_{\mathscr{V}}(A)} \Sk_{\mathscr{V}}(N)\) is the Tambara relative tensor product of \(k\)-linear categories (see \cref{defn:reltensorproduct}). In \cref{thm:relaresame} we prove that 
    \[ 
    \Sk_{\mathscr{V}}(M) \otimes_{\Sk_{\mathscr{V}}(A)} \Sk_{\mathscr{V}}(N) \simeq \Sk_{\mathscr{V}}(M) \ftensor{\otimes}_{\Sk_{\mathscr{V}}(A)} \Sk_{\mathscr{V}}(N).
    \]
    Hence, 
    \[
    \Sk_{\mathscr{V}}\left( M \sqcup_{A} N \right) \simeq \Sk_{\mathscr{V}}(M) \ftensor{\otimes}_{\Sk_{\mathscr{V}}(A)} \Sk_{\mathscr{V}}(N)
    \]
    and skein categories satisfy the excision of factorisation homologies.
\end{enumerate}
As a factorisation homology is fully characterised by the above (\cref{thm:factcharacterised}), this concludes the proof.
\end{proof}

\section{Quantised Character Varieties}
\label{section:character-varieties}
In this final section we shall give an application to the quantisation of character varieties. 

\begin{defn}
Let \(G\) be a reductive Lie group. The \emph{representation variety} \(\mathfrak{R}_G(\Sigma)\) is the affine variety 
\[
\mathfrak{R}_G(\Sigma) = \left\{\,\rho: \pi_1(\Sigma) \to G \,\right\} 
\]
of homomorphisms from the fundamental group of \(\Sigma\) to \(G\).
\end{defn}

\begin{defn}
The \emph{character variety} \(\Chg_G(\Sigma)\) is the affine categorical quotient \(\mathfrak{R}_G(\Sigma) // G\) of the representation variety of the surface \(\mathfrak{R}_G(\Sigma)\) by the group \(G\) acting upon it by conjugation. The character variety \(\Chg_G(\Sigma)\) has a Poisson structure \cite{AB83, Goldman84} so by a quantisation of \(\Chg_G(\Sigma)\) we mean a deformation of this Poisson structure.
\end{defn}

We shall prove that the skein algebra \(\skalg_{\Repfd_q(G)}(\Sigma)\), of a punctured surface \(\Sigma\) coloured by the category of finite-dimensional representations of the quantum group \(\qgroup{\frg}\) where q is a generic parameter, is a deformation quantisation of the character variety \(\Chg_G(\Sigma)\). 

In order to do this we shall first use the characterisation of skein categories as \(k\)-linear factorisation homology (\cref{thm:skeinklinear}) to prove that the free cocompletion of the skein category \(\Sk_{\mathscr{V}}(\Sigma)\) is the \(\LFP{k}\) factorisation homology \(\int_{\Sigma} \Free(\mathscr{V})\) (\cref{thm:freecompskeincat}).  We shall then use the results of Ben-Zvi, Brochier and Jordan \cite{david1} who showed that one can acquire quantisations of character varieties via \(\LFP{\mathbb{C}}\) factorisation homologies with coefficients in \(\Rep_q(G)\).

\subsection{The Category \texorpdfstring{\(\LFP{k}\)}{LFPk}}
Firstly, we define the category \(\LFP{k}\).

\begin{defn}
A category is \emph{locally finitely presentable} if it is locally small, cocomplete and is generated under filtered colimits by a set of compact objects. 
\end{defn}

\begin{rmk}
A category \(\mathscr{C}\) is \emph{locally finitely presentable} if it is \(\aleph_0\)-locally presentable. 
\end{rmk}

\begin{defn}
A functor \(F: \mathscr{C} \to \mathscr{D}\) is \emph{compact} if it preserves compact objects i.e.\ if \(c\) is a compact object of \(\mathscr{C}\) then \(F(c)\) is a compact object of \(\mathscr{D}\).
\end{defn}

\begin{defn}
Let \(\LFP{k}\) be the \((2,1)-category\) with
\begin{enumerate}
    \item objects: locally finitely presentable \(k\)-linear categories\footnote{Note that in \(\LFP{k}\) the categories are not assumed to be small whereas in \(\Cat_k\) they are.};
    \item \(1\)-morphisms: compact cocontinuous \(k\)-linear functors;
    \item \(2\)-morphisms: natural isomorphisms.
\end{enumerate}
\end{defn}

The \((2,1)\)-category \(\LFP{k}\) is a strict monoidal category with the Kelly--Deligne tensor product \(\boxtimes\) as the monoidal product\footnote{The monoidal unit of \(\LFP{k}^{\boxtimes}\) is \(\kMod\).}.

\begin{defn}
Let \(\mathscr{A}, \mathscr{B}, \mathscr{C} \in \LFP{k}\), \(\Cocont(\mathscr{A} \boxtimes \mathscr{B}, \mathscr{C})\) be the category of cocontinuous functors \(\mathscr{A} \boxtimes \mathscr{B} \to \mathscr{C}\) and \( \Cocont(\mathscr{A}, \mathscr{B}; \mathscr{C})\) be the category of bilinear functors \(\mathscr{A} \times \mathscr{B} \to \mathscr{C}\) which are cocontinuous in each variable separately.
The \emph{Kelly--Deligne tensor product} of \(\mathscr{A}\) and \(\mathscr{B}\) is the category \(\mathscr{A} \boxtimes \mathscr{B} \in \LFP{k}\) together with a bilinear functor \(S: \mathscr{A} \times \mathscr{B} \to \mathscr{A} \boxtimes \mathscr{B} \in \Cocont(\mathscr{A}, \mathscr{B}; \mathscr{C})\) such that composition with \(S\) defines an equivalence of categories 
\[
\Cocont(\mathscr{A} \boxtimes \mathscr{B}, \mathscr{C}) \simeq \Cocont(\mathscr{A}, \mathscr{B}; \mathscr{C}) \cong \Cocont(\mathscr{A}, \Cocont(\mathscr{B}, \mathscr{C}))
\]
for all \(\mathscr{C} \in \LFP{k}\).
\end{defn}

\begin{rmk}
Kelly \cite[Proposition~4.3]{Kelly82} proved the existence of \(\mathscr{A} \boxtimes \mathscr{B}\) for categories \(\mathscr{A}, \mathscr{B} \in \Rex\) where \(\Rex\) is the \((2,1)\)-category of essentially small, finitely cocomplete categories with right exact functors as \(1\)-morphisms and natural isomorphisms as \(2\)-morphisms. Franco \cite[Theorem~18]{Franco2012} showed that for the abelian categories \(\mathscr{A}\)  and \(\mathscr{B}\), their Kelly tensor product \(\mathscr{A} \boxtimes \mathscr{B}\) is the Deligne tensor product of abelian categories \cite{Deligne90} when the Deligne tensor product exists; hence, the name Kelly--Deligne tensor product. For the existence of the Kelly--Deligne tensor product in \(\LFP{k}\) see \cite[Section 2.4.1]{TensorThesis} and the references therewithin. 
\end{rmk}

As the category \(\LFP{k}^{\boxtimes}\) is \(\boxtimes\)-presentable  \cite[Proposition~3.5]{david1}, one can define \(\LFP{k}\) factorisation homology.
\subsection{The \texorpdfstring{\(\LFP{k}\)}{LFPk} Factorisation Homology of Punctured Surfaces}

We shall now very briefly recall the relevant results of Ben-Zvi, Brochier and Jordan \cite{david1} as to how to acquire quantisations of character varieties via \(\LFP{\mathbb{C}}\) factorisation homology. For a fuller account see \cite{david1} or \cite[Section~3.1, Section~3.3]{CookeThesis}.
Throughout this section let \(\mathscr{E} \in \LFP{k}\) be a rigid abelian balanced braided monoidal \(k\)-linear category. 
\begin{ex}
Let \(G\) be a connected Lie group such that its Lie group \(\frg = \Lie(G)\) is semisimple. The category \(\Repfd_q(G)\) of finite-dimensional, integrable \(\qgroup{\frg}\)-modules is a \(\mathbb{C}\)-linear ribbon category, so is a suitable choice of \(\mathscr{E}\). Another choice of \(\mathscr{E}\) is the category of integrable representations \(\Rep_q(G)\) whose objects are possibly infinite direct sums of simple modules in \(\Repfd_q(G)\).
\end{ex}

Firstly note that, the factorisation homology \(\int_{\Sigma}^{\LFP{k}} \mathscr{E}\) of the punctured surface \(\Sigma\) can be given the structure of an \(\mathscr{E}\)-module category as follows:
\begin{figure}[H]
    \centering
    \myincludesvg{width=0.66\textwidth}{}{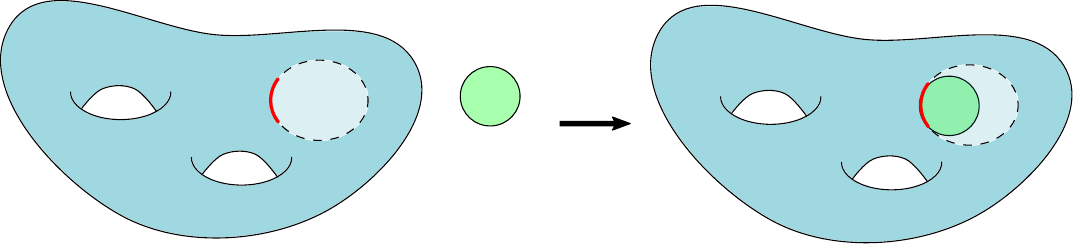}
    \caption{An illustration of the map \(\Sigma \sqcup \mathbb{D} \to \Sigma\). The surface \(\Sigma_{2,1}\) has a interval marked in red along its boundary along which the disc \(\mathbb{D}\) is attached. The resultant surface is isotopic to \(\Sigma_{2,1}\). }
    \label{fig:2}
\end{figure}
\noindent
Choose an interval along the boundary of \(\Sigma\)\footnote{The module structure depends on the choice of marking.}. The mapping
\(\Sigma \sqcup \mathbb{D} \to \Sigma,\)
which attaches the disc \(\mathbb{D}\) to \(\Sigma\) along the marked interval, induces a \(\int_{\mathbb{D}} \mathscr{E}\)-module structure on \(\int_{\Sigma} \mathscr{E}\). As \(\int_{\mathbb{D}} \mathscr{E} \simeq \mathscr{E}\) in \(\LFP{k}\), this means that \(\int_{\Sigma} \mathscr{E}\) is a \(\mathscr{E}\)-module. This module structure defines the action map 
\[\act_m: \mathscr{E} \to \int^{\LFP{k}}_{\Sigma} \mathscr{E}: a \mapsto a \cdot m\]
for any \(m \in \int^{\LFP{k}}_{\Sigma} \mathscr{E}\), and this action map has a right adjoint 
\[\act_m^R: \int^{\LFP{k}}_{\Sigma} \mathscr{E} \to \mathscr{E}.\]

Not only is \(\int_{\Sigma} \mathscr{E}\) an \(\mathscr{E}\)-module category, but it is the category of modules of an algebra \(A_{\Sigma}\) in \(\mathscr{E}\). This algebra \(A_{\Sigma}\) is an internal Hom:

\begin{defn}[{\cite[147]{EtingofTensorCategories}}\footnote{Note that Etignof et al.\ are assuming that \(\mathscr{M}\) is a \emph{left} \(\mathscr{E}\)-module category, whereas we are assuming that is is a \emph{right} \(\mathscr{E}\)-module category. Also note that they assume the categories are finite, but the proofs work without modification for locally finitely presentable categories.}]
Let \(\mathscr{M}\) be a right \(\mathscr{E}\)-module category and let \(m, n \in \mathscr{M}\). The \emph{internal Hom}\footnote{Also known as the enriched Hom.} from \(m\) to \(n\) is the object \(\IHom(m, n) \in \mathscr{E}\) which represents the functor \(x \mapsto \Hom_{\mathscr{M}}(m \cdot x, n)\) i.e.\ such that there is a natural isomorphism 
\[\eta^{m, n}: \Hom_{\mathscr{M}}(m \cdot \_\;, n) \xrightarrow{\sim} \Hom_{\mathscr{E}}(\_\;, \IHom(m, n)).\]
For any \(m \in \mathscr{M}\), \(\End(m) := \IHom(m, m)\) is an algebra in \(\mathscr{E}\)~\cite[149]{EtingofTensorCategories} called the \emph{internal endomorphism algebra of \(m\)}.
\end{defn}

\begin{defn}
The embedding \(\emptyset \xhookrightarrow{} \Sigma\) induces a pointing map \(\kMod \to \int_{\Sigma} \mathscr{E}\) on the factorisation homology \(\int_{\Sigma}^{\LFP{k}} \mathscr{E}\).
The \emph{distinguished object} \(\mathscr{O}_{\mathscr{E}, \Sigma}\) of the factorisation homology \(\int_{\Sigma}^{\LFP{k}} \mathscr{E}\) is the image of \(k\) under this pointing map \(\kMod \to \int_{\Sigma} \mathscr{E}\).
\end{defn}

\begin{defn}
\label{defn:facthomalg}
The\footnote{The algebra object is dependent on the choice marking of \(\Sigma\)} \emph{algebra object} \(A_{\Sigma}\) of \(\int_{\Sigma}^{\LFP{k}} \mathscr{E}\) is the internal endomorphism algebra of the distinguished object
\[A_{\Sigma} := \End_{\mathscr{E}}(\mathscr{O}_{\mathscr{E}, \Sigma}).\]
This is called the \emph{moduli algebra} of \(\Sigma\) in \cite{david1}.
\end{defn}

\begin{prop}{\cite[Theorem~5.14]{david1}\footnote{In Theorem 5~.11, it states that \(\mathscr{E}\) is a rigid braided monoidal category rather than a balanced monoidal category. Any braided monoidal category constructs a (non-framed) \(E_2\)-algebra which are the coefficients for the framed version of factorisation homology. As we are working with the oriented version of factorisation homology we need a framed \(E_2\)-algebra, and hence \(\mathscr{E}\) must be balanced monoidal.}}
\label{prop:factmodcat}
Let \(\Sigma\) be a punctured surface and \(A_{\Sigma}\) be the algebra object of the factorisation homology \(\int^{\LFP{k}}_{\Sigma} \mathscr{E}\). There is an equivalence of categories
\[
    \int_{\Sigma} \mathscr{E} \simeq \Mod{A_{\Sigma}}{\mathscr{E}},
\]
where \(\Mod{A_{\Sigma}}{\mathscr{E}}\) is the category of right modules over \(A_{\Sigma}\) in \(\mathscr{E}\).
\end{prop}

\begin{rmk}
Note that as the factorisation homology is equivalent to a category of modules of an algebra, it is an abelian category. 
\end{rmk}
\begin{rmk}
There is a combinatorial description of \(A_{\Sigma}\) in terms of the gluing pattern of the surface (see \cite[Section~5]{david1}). 
\end{rmk}

Now assume \(\mathscr{E} = \Rep_q(G)\), the category of integrable representation of \(\qgroup{\frg}\). As \(A_{\Sigma} \in \Rep_q(G)\) it is a \(\qgroup{\frg}\)-module. Hence, there is an action of the Hopf algebra \(\qgroup{\frg}\) on \(A_{\Sigma}\).
\begin{defn}
\label{defn:invalg}
The \emph{algebra of invariants} \(\mathscr{A}_{\Sigma}\) of the punctured surface \(\Sigma\) with respect to the quantum group \(\qgroup{\frg}\) is \(A_{\Sigma}^{\qgroup{\frg}}\), the algebra of invariants of \(A_{\Sigma}\) under the action of \(\qgroup{\frg}\). In other words, \(\mathscr{A}_{\Sigma}\) is the subgroup of \(x \in A_{\Sigma}\) such that \(g \cdot x = \epsilon(g)x\) for all \(g \in \qgroup{\frg}\) where \(\epsilon\) is the unit of the Hopf algebra \(\qgroup{\frg}\).
\end{defn}
To quantise the character variety \(\Chg_G(\Sigma)\) one deforms the Poisson algebra of functions on \(\Chg_G(\Sigma)\), and a suitable deformation of this Poisson algebra is given by \(\mathscr{A}_{\Sigma}\):
\begin{thm}[{\cite[Theorem~7.3]{david1}}]
\label{thm:invalgquantisation}
Let \(\Sigma\) be a punctured surface. The algebra of invariants \(\mathscr{A}_{\Sigma}\) of \(\int_{\Sigma} \Rep_q(G)\) is a deformation quantisation of the character variety \(\Chg_G(\Sigma)\).
\end{thm}
\subsection{Cauchy and Free Completions}
We shall now define free cocompletions, Cauchy cocompletions and the subcategory of compact projectives. These are needed to state \cref{thm:freecompskeincat}.
In this subsection let \(\mathscr{V}\) be a closed symmetric monoidal category with all its limits and colimits e.g.\ \(\mathscr{V} = \catname{Set}\) or \(\mathscr{V} = \kMod\).
Let \(\mathscr{C}, \mathscr{D}, \mathscr{E}\) denote small \(\mathscr{V}\)-enriched categories and let \(\mathscr{X}\) denote a not necessarily small \(\mathscr{V}\)-enriched category. We begin by defining free cocompletion:

\begin{defn}
Let \(\mathscr{C}\) be a small \(\mathscr{V}\)-enriched category. The \emph{free cocompletion} \(\Free(\mathscr{C})\) is the enriched functor category \([\mathscr{C}^{op}, \mathscr{V}]\).
\end{defn}

\begin{prop}[{\cite[44]{Kelly82}}]
The free cocompletion \(\Free(\mathscr{C})\) of the small \(\mathscr{V}\)-enriched category \(\mathscr{C}\) is locally finitely presentable i.e.\ \(\Free(\mathscr{C}) \in \LFP{k}\). 
\end{prop}

Note that there is an equivalence
\[
[\mathscr{C}^{op}, \mathscr{V}] \simeq \catname{Prof}_{\mathscr{V}}(\catname{Pt}, \mathscr{C})
\] 
where \(\catname{Prof}_{\mathscr{V}}\) is the \((2,1)\)-category of \(\mathscr{V}\)-enriched profunctors and \(\catname{Pt}\) is the one-point category defined in \cref{sec:catkdefn}. This is needed to define the Cauchy cocompletion as the Cauchy cocompletion is defined as a subcategory of \(\catname{Prof}_{\mathscr{V}}\) rather than of \([\mathscr{C}^{op}, \mathscr{V}]\). 

\begin{defn}
A \(\mathscr{V}\)-enriched category \(\mathscr{C}\) is \emph{Cauchy complete} if it has all absolute colimits.
\end{defn}

\begin{ex}
When \(\mathscr{V} = \kMod\), a category \(\mathscr{C}\) is Cauchy complete if and only if it is idempotent complete and has all direct sums.
\end{ex}

\begin{defn}
Let \(\mathscr{C}\) be a small \(\mathscr{V}\)-enriched category. The \emph{Cauchy completion} \(\cauchy(\mathscr{C})\) of \(\mathscr{C}\) is the full subcategory of \(\catname{Prof}_{\mathscr{V}}(\catname{Pt}, \mathscr{C})\) consisting of profunctors which have a right adjoint in \(\catname{Prof}_{\mathscr{V}}\).
\end{defn}

The Cauchy completion \(\cauchy(\mathscr{C})\) is in fact a Cauchy complete category due to:
\begin{prop}[{\cite[378]{Street83}}]
Every colimit weighted by the \(\mathscr{V}\)-enriched profunctor \(\varphi\) is absolute if and only if \(\varphi\) has a right adjoint in \(\catname{Prof}_{\mathscr{V}}\).
\end{prop}

We now define the subcategory of compact projective objects.

\begin{defn}
Let \(\mathscr{X}\) be a \(\mathscr{V}\)-enriched category. An object \(x \in \mathscr{X}\) is \emph{compact projective}\footnote{Compact projective objects are also known as small projective objects or tiny objects.} if the corepresentable functor \(\mathscr{X}(x, \_): \mathscr{X} \to \mathscr{V}\) preserves all small colimits.
\end{defn}

\begin{defn}
Let \(\mathscr{X}\) be a \(\mathscr{V}\)-enriched category. The full subcategory of compact projective objects of \(\mathscr{X}\) is denoted \(\Comp(\mathscr{X})\).
\end{defn}

Finally, \(\Free, \cauchy\) and \(\Comp\) can be related as follows:

\begin{prop}[{\cite[95]{Kelly05}}]
\label{prop:compoffree}
Let \(\mathscr{C}\) be a small \(\mathscr{V}\)-enriched category. Then
\[\Comp \left( \Free (\mathscr{C}) \right) \simeq \cauchy(\mathscr{C}).\]
\end{prop}

\subsection{Free Cocompletions of Skein Categories are \(\LFP{k}\) Factorisation Homologies}
We now use the characterisation of skein categories as \(k\)-linear factorisation homology to relate skein categories to \(\LFP{k}\) factorisation homology.

\begin{thm}
\label{thm:freecompskeincat}
Let \(\mathscr{A}\) be \(k\)-linear strict ribbon category and let \(\Sigma\) be a surface.
There are equivalences of categories
\begin{align*}
\Free (\Sk_{\mathscr{A}}(\Sigma)) &\simeq \int^{\LFP{k}}_{\Sigma} \Free(\mathscr{A}) \\
\text{ and } \cauchy(\Sk_{\mathscr{A}}(\Sigma)) &\simeq \Comp\left(\int^{\LFP{k}}_{\Sigma} \Free(\mathscr{A}) \right).
\end{align*}
\begin{proof}
We begin with the first equivalence. By \cref{thm:skeinklinear},
\begin{align*}
    \Free(\Sk_{\mathscr{A}}(\Sigma)) &\simeq \Free \left( \int^{\Cat_k}_{\Sigma} \mathscr{A} \right)\\
    &\simeq \int^{\LFP{k}}_{\Sigma} \Free(\mathscr{A})
\end{align*}
as free cocompletion preserves Kan extensions \cite[Section~4.1]{Kelly05}. For the second equivalence apply \(\Comp\):
\begin{align*}
    &\Comp \left( \Free(\Sk_{\mathscr{A}}(\Sigma)) \right) \simeq \Comp \left( \int^{\LFP{k}}_{\Sigma} \Free(\mathscr{A}) \right) \\
    \implies &\cauchy \left( \Sk_{\mathscr{A}}(\Sigma) \right) \simeq \Comp \left( \int^{\LFP{k}}_{\Sigma} \Free(\mathscr{A}) \right)
\end{align*}
by \cref{prop:compoffree}.
\end{proof}
\end{thm}
\subsection{Skein Algebras are Quantisations of Character Varieties}
We conclude this paper by using \cref{thm:freecompskeincat} to prove that skein algebras of punctured surfaces with generic parameters correspond to the algebras of invariants of \(\LFP{k}\) factorisation homology. Hence, these skein algebras are deformation quantisation of character varieties. 

\begin{defn}
Let \(\mathscr{A}\) be a strict \(k\)-linear ribbon category and let \(\Sigma\) be an oriented surface. The \emph{skein algebra} 
\[
\skalg_{\mathscr{V}}(\Sigma) :=   \operatorname{End}_{\Sk_{\mathscr{V}}(\Sigma)}(\emptyset) = \operatorname{End}_{\cauchy(\Sk_{\mathscr{V}}(\Sigma))}(\emptyset)\footnote{As \(\Sk(\Sigma)\) is a full subcategory of \(\Cauchy(\Sk(\Sigma))\).}
\]
where the distinguished object of \(\Sk_{\mathscr{V}}(\Sigma)\) is the empty set.
\end{defn}

\begin{thm}
\label{thm:skalgalginv}
Let \(\Sigma\) be a punctured surface and \(q \in \mathbb{C}^{\times}\) be \(1\) or not a root of unity. There is an algebra isomorphism 
\[
\skalg_{\Repfd_q(G)}(\Sigma) \cong \mathscr{A}_{\Sigma}
\]
between the skein algebra and the algebra of invariants \(\mathscr{A}_{\Sigma}\) of  \(\int^{\LFP{\mathbb{C}}}_{\Sigma} \Rep_q(G)\). 
\begin{proof}
We begin with the right hand side. For brevity we shall denote \(Z_q(\Sigma) := \int^{\LFP{\mathbb{C}}}_{\Sigma} \Rep_q(G)\).
By definition \(\mathscr{A}_{\Sigma} := A_{\Sigma}^{\qgroup{\frg}}\), the algebra of invariants of the algebra object \(A_{\Sigma} := \End(\mathscr{O})\) where \(\mathscr{O}\) is the distinguished object of \(Z_q(\Sigma)\). Hence,
\[
\mathscr{A}_{\Sigma} = A_{\Sigma}^{\qgroup{\frg}} = \Hom_{\Rep_q(G)}(\mathscr{O} , A_S) = \Hom_{\Rep_q(G)}(\mathscr{O} , \End(\mathscr{O})).
\]
As \(\End(m) = \act_m^R(m)\) \cite[Definition~3.9]{david1}\footnote{Ben-Zvi, Brochier and Jordan define \(\End(m) := \act_m^R(m)\), but this definition is equivalent to ours as \(\End(m)\) is unique up to isomorphism and \(\act_m^R(m)\) satisfies the required universal property.},
\[
\mathscr{A}_{\Sigma} = \Hom_{\Rep_q(G)}(\mathscr{O} , \End(\mathscr{O})) = \Hom_{\Rep_q(G)}(\mathscr{O} , \act_{\mathscr{O}}^R(\mathscr{O})) 
\]
By the definition of the right adjoint
\[\mathscr{A}_{\Sigma} = \Hom_{\Rep_q(G)}(\mathscr{O} , \act_{\mathscr{O}}^R(\mathscr{O})) = \Hom_{Z_q(\Sigma)}(\act_{\mathscr{O}}(\mathscr{O}), \mathscr{O}) = \operatorname{End}_{Z_q(\Sigma)}(\mathscr{O}).\]
As \(q\) is generic, the distinguished object \(\mathscr{O}\) of \(Z_q(\Sigma)\) is compact projective and \(\Comp(Z_q(\Sigma))\) is a full subcategory; hence,
\[\mathscr{A}_{\Sigma} = \operatorname{End}_{Z_q(\Sigma)}(\mathscr{O}) =  \operatorname{End}_{\Comp\left(Z_q(\Sigma) \right)}(\mathscr{O}).\]
By \cref{thm:freecompskeincat}, 
\[\Comp\left(Z_q(\Sigma) \right) \simeq \cauchy\left(\Sk_{\Repfd_q(G)}(\Sigma)\right).\]
So, we conclude that
\[\mathscr{A}_{\Sigma} = \operatorname{End}_{\Comp\left(Z_q(\Sigma) \right)}(\mathscr{O}) \simeq  \operatorname{End}_{\cauchy(\Sk(\Sigma))}(\emptyset) = \skalg(\Sigma).\]
\end{proof}
\end{thm}

\begin{cor}
\label{cor:skeincharacter}
Let \(G\) be any connected Lie group such that \(\Lie(G) = \frg\) is semisimple, \(q \in \mathbb{C}^{\times}\) be \(1\) or not a root of unity, and \(\Sigma\) be a punctured surface. Then \(\skalg_{\Repfd_q(G)}(\Sigma)\) is a deformation quantisation of the character variety \(\Chg_G(\Sigma)\).
\begin{proof}
By \cref{thm:skalgalginv}, \(\skalg_{\Repfd_q(G)}(\Sigma) \simeq \mathscr{A}_{\Sigma}\) and, by \cref{thm:invalgquantisation}, \(\mathscr{A}_{\Sigma}\) is a deformation quantisation of the character variety \(\Chg_G(\Sigma)\).
\end{proof}
\end{cor}

\printbibliography[heading=bibintoc,title={References}]
\end{document}